\documentclass[12pt]{article}
\usepackage[margin=1in]{geometry}
\usepackage{tikz}
\usepackage{amsmath,amssymb,amsthm}
\usepackage[hidelinks]{hyperref}
\usepackage{cprotect}

\allowdisplaybreaks

\newtheorem{theorem}{Theorem}[section]
\newtheorem{proposition}[theorem]{Proposition}

\newtheorem{lemma}[theorem]{Lemma}

\theoremstyle{definition}

\newtheorem{question}{Question}[section]
\newtheorem{conjecture}{Conjecture}[section]
\newtheorem{definition}{Definition}[section]

\DeclareMathOperator{\dist}{dist}
\DeclareMathOperator*{\coal}{\circ}

\def\PICcoalesce{
\begin{tabular}{c@{\qquad}c@{\qquad}c}
\begin{tabular}{c}
\begin{tikzpicture}[scale = 1.25]
\tikzstyle{Vertex}=[inner sep=1.5pt,thick,shape=circle,draw=black,fill=white];
\tikzstyle{Vertex2}=[inner sep=2pt,thick,shape=rectangle, rounded corners,draw=black,fill=white];
\tikzstyle{Edge}=[thick];

\draw[rounded corners, color=black!30!white,ultra thick, dashed] (-0.5,0.7) rectangle (0.5,3.3)

(-0.5,0.3) rectangle (0.5,-0.3)

(-1.25,-0.7) rectangle (1.25,-1.3);

\node[Vertex2] (a) at (0,3) {\small$1{:}1{:}1$};
\node[Vertex2] (b) at (0,2) {\small$1{:}1{:}2$};
\node[Vertex2] (c) at (0,1) {\small$1{:}1{:}3$};
\node[Vertex2] (d) at (0,0) {\small$2{:}1{:}1$};
\node[Vertex2] (e) at (-0.75,-1) {\small$3{:}1{:}1$};
\node[Vertex2] (f) at (0.75,-1) {\small$3{:}1{:}2$};

\draw[Edge] (a)--(b)--(c)--(d)--(e)--(f)--(d);

\end{tikzpicture}\end{tabular}& 
\begin{tabular}{l}
\begin{tikzpicture}[scale = 1.25]
\tikzstyle{Vertex}=[inner sep=1.5pt,thick,shape=circle,draw=black,fill=white];
\tikzstyle{Edge}=[thick];
\node[Vertex] (a) at (0,0) {$1$};
\node[Vertex] (b) at (0.75,0) {$2$};
\node[Vertex] (c) at (1.5,0) {$3$};
\draw[Edge] (a)--(b)--(c);
\node[left] at (-0.25,0) {$H_1$:};
\end{tikzpicture}\\[10pt]
\begin{tikzpicture}[scale = 1.25]
\tikzstyle{Vertex}=[inner sep=1.5pt,thick,shape=circle,draw=black,fill=white];
\tikzstyle{Edge}=[thick];
\node[Vertex] (a) at (0,0) {$1$};
\node[left] at (-0.25,0) {$H_2$:};
\end{tikzpicture}\\[10pt]
\begin{tikzpicture}[scale = 1.25]
\tikzstyle{Vertex}=[inner sep=1.5pt,thick,shape=circle,draw=black,fill=white];
\tikzstyle{Edge}=[thick];
\node[Vertex] (a) at (0,0) {$1$};
\node[Vertex] (b) at (0,-2) {$2$};
\node[Vertex] (c) at (0.5,-1) {$3$};
\draw[Edge] (a)--(b)--(c)--(a);
\node[left] at (-0.25,-1) {$H_3$:};
\end{tikzpicture}
\end{tabular}
&
\begin{tabular}{c}
\begin{tikzpicture}[scale = 0.8]
\tikzstyle{Vertex}=[inner sep=3.5pt,thick,shape=circle,draw=black,fill=white];
\tikzstyle{Edge}=[thick];

\draw[rounded corners, color=black!30!white,ultra thick, dashed] 
    (-0.4,0.6) rectangle (0.4,3.4)
    (0.9,0.6) rectangle (1.7,3.4)
    (2.2,0.6) rectangle (3,3.4)
    (-0.4,0.4) rectangle (0.4,-0.4)
    (-1.15,-0.6) rectangle (1.15,-1.4)
    (-0.65,-1.6) rectangle (1.65,-2.4)
    (-1.15,-2.6) rectangle (1.15,-3.4)
;

\node[above] at (0,3.4) {\small$1{:}1$};
\node[above] at (1.3,3.4) {\small$1{:}2$};
\node[above] at (2.6,3.4) {\small$1{:}3$};
\node[right] at (0.4,0) {\small$2{:}1$};
\node[right] at (1.15,-1) {\small$3{:}1$};
\node[right] at (1.65,-2) {\small$3{:}3$};
\node[right] at (1.15,-3) {\small$3{:}2$};

\node[Vertex] (a1) at (0,3) {};
\node[Vertex] (b1) at (0,2) {};
\node[Vertex] (c1) at (0,1) {};
\node[Vertex] (a2) at (1.3,3) {};
\node[Vertex] (b2) at (1.3,2) {};
\node[Vertex] (c2) at (1.3,1) {};
\node[Vertex] (a3) at (2.6,3) {};
\node[Vertex] (b3) at (2.6,2) {};
\node[Vertex] (c3) at (2.6,1) {};
\node[Vertex] (d1) at (0,0) {};
\node[Vertex] (e1) at (-0.75,-1) {};
\node[Vertex] (f1) at (0.75,-1) {};
\node[Vertex] (e2) at (-0.75,-3) {};
\node[Vertex] (f2) at (0.75,-3) {};
\node[Vertex] (e3) at (-0.25,-2) {};
\node[Vertex] (f3) at (1.25,-2) {};

\draw[Edge] (a1)--(b1)--(c1)--(d1)--(e1)--(f1)--(d1)
(a1)--(a2)--(a3) (b1)--(b2)--(b3) (c1)--(c2)--(c3) (e1)--(e2)--(e3)--(e1) (f1)--(f2)--(f3)--(f1);

\end{tikzpicture}
\end{tabular}
\\[10pt]
(a) $G$&
(b) $H_1,H_2,H_3$&
(c) $\displaystyle G\coal_{i=1}^3{}_{V_i}H_i$
\end{tabular}
}

\def\PICAA{
\begin{tikzpicture}[scale = 1]
\tikzstyle{Vertex}=[inner sep=1.5pt,thick,shape=circle,draw=black,fill=white];
\tikzstyle{Edge}=[thick];
\node[Vertex] (1) at (0.5,0.87) {\scriptsize$\mathbf1$};
\node[Vertex] (2) at (1.5,0.87) {\scriptsize$\mathbf2$};
\node[Vertex] (3) at (2,0) {\scriptsize$\mathbf3$};
\node[Vertex] (4) at (1.5,-0.87) {\scriptsize$\mathbf4$};
\node[Vertex] (5) at (0.5,-0.87) {\scriptsize$\mathbf5$};
\node[Vertex] (6) at (0,0) {\scriptsize$\mathbf6$};
\node[Vertex] (7) at (1,0) {\scriptsize$\mathbf7$};
\draw[Edge] (2)--(7)--(1)--(6)--(7)--(4)--(3) (6)--(5)--(7);
\end{tikzpicture}
}

\def\PICAB{
\begin{tikzpicture}[scale = 1]
\tikzstyle{Vertex}=[inner sep=1.5pt,thick,shape=circle,draw=black,fill=white];
\tikzstyle{Edge}=[thick];
\node[Vertex] (1) at (-0.5,-0.87) {\scriptsize$\mathbf1$};
\node[Vertex] (2) at (-0.5,0.87) {\scriptsize$\mathbf2$};
\node[Vertex] (3) at (0.5,0.87) {\scriptsize$\mathbf3$};
\node[Vertex] (4) at (1,0) {\scriptsize$\mathbf4$};
\node[Vertex] (5) at (1.5,-0.87) {\scriptsize$\mathbf5$};
\node[Vertex] (6) at (0.5,-0.87) {\scriptsize$\mathbf6$};
\node[Vertex] (7) at (0,0) {\scriptsize$\mathbf7$};
\draw[Edge] (2)--(7)--(6)--(4)--(7)--(3)--(4) (1)--(6)--(5);
\end{tikzpicture}
}

\def\PICBA{
\begin{tikzpicture}[scale = 1]
\tikzstyle{Vertex}=[inner sep=1.5pt,thick,shape=circle,draw=black,fill=white];
\tikzstyle{Edge}=[thick];
\node[Vertex] (1) at (1,-1) {\scriptsize$\mathbf1$};
\node[Vertex] (2) at (3,0) {\scriptsize$\mathbf2$};
\node[Vertex] (3) at (1,1) {\scriptsize$\mathbf3$};
\node[Vertex] (4) at (2,0) {\scriptsize$\mathbf4$};
\node[Vertex] (5) at (1,0) {\scriptsize$\mathbf5$};
\node[Vertex] (6) at (2,1) {\scriptsize$\mathbf6$};
\node[Vertex] (7) at (2,-1) {\scriptsize$\mathbf7$};
\node[Vertex] (8) at (0,0) {\scriptsize$\mathbf8$};
\draw[Edge] (3)--(5)--(8)--(3)--(4)--(5)--(7)--(1)--(8) (1)--(4)--(6)--(3) (3)--(2)--(7)--(4)--(2);
\end{tikzpicture}
}

\def\PICBB{
\begin{tikzpicture}[scale = 1]
\tikzstyle{Vertex}=[inner sep=1.5pt,thick,shape=circle,draw=black,fill=white];
\tikzstyle{Edge}=[thick];
\node[Vertex] (1) at (1.5,-1) {\scriptsize$\mathbf1$};
\node[Vertex] (2) at (2,1) {\scriptsize$\mathbf2$};
\node[Vertex] (3) at (1,1) {\scriptsize$\mathbf3$};
\node[Vertex] (4) at (1,0) {\scriptsize$\mathbf4$};
\node[Vertex] (5) at (0,0) {\scriptsize$\mathbf5$};
\node[Vertex] (6) at (0.5,-1) {\scriptsize$\mathbf6$};
\node[Vertex] (7) at (0,1) {\scriptsize$\mathbf7$};
\node[Vertex] (8) at (2,0) {\scriptsize$\mathbf8$};
\draw[Edge] (3)--(7)--(5)--(3)--(4)--(5)--(6)--(4)--(8)--(2)--(5)--(1)--(8)--(3) (1)--(6) (2)--(4);
\end{tikzpicture}
}

\def\PICCA{
\begin{tikzpicture}[scale = 0.8]
\tikzstyle{Vertex}=[inner sep=2.5pt,thick,shape=circle,draw=black,fill=white];
\tikzstyle{Edge}=[thick];

\node[Vertex] (a) at (0,0) {\scriptsize$\mathbf1$};
\node[Vertex,inner sep = 1.25pt] (b) at (-0.5,1) {\scriptsize$\mathbf{12}$};
\node[Vertex] (c) at (-1,2) {\scriptsize$\mathbf3$};
\node[Vertex] (d) at (-1.5,3) {\scriptsize$\mathbf4$};
\node[Vertex] (e) at (0,2) {\scriptsize$\mathbf5$};
\node[Vertex] (f) at (-.5,3) {\scriptsize$\mathbf6$};
\node[Vertex] (g) at (0.5,1) {\scriptsize$\mathbf7$};
\node[Vertex] (h) at (1,2) {\scriptsize$\mathbf8$};
\node[Vertex] (i) at (0.5,3) {\scriptsize$\mathbf9$};
\node[Vertex,inner sep = 1.25pt] (j) at (0,4) {\scriptsize$\mathbf{10}$};
\node[Vertex,inner sep = 1.25pt] (k) at (-.5,5) {\scriptsize$\mathbf{11}$};
\node[Vertex] (l) at (1.5,3) {\scriptsize$\mathbf2$};
\node[Vertex,inner sep = 1.25pt] (m) at (2,4) {\scriptsize$\mathbf{13}$};
\node[Vertex,inner sep = 1.25pt] (n) at (2.5,5) {\scriptsize$\mathbf{14}$};
\node[Vertex,inner sep = 1.25pt] (o) at (2,6) {\scriptsize$\mathbf{15}$};
\node[Vertex,inner sep = 1.25pt] (p) at (3,6) {\scriptsize$\mathbf{16}$};

\draw[Edge] (a)--(b)--(c)--(d);
\draw[Edge] (b)--(e)--(f);
\draw[Edge] (a)--(g)--(h)--(i)--(j)--(k);
\draw[Edge] (h)--(l)--(m)--(n);
\draw[Edge] (p)--(n)--(o);
\end{tikzpicture}
}

\def\PICCB{
\begin{tikzpicture}[scale = 0.8]
\tikzstyle{Vertex}=[inner sep=2.5pt,thick,shape=circle,draw=black,fill=white];
\tikzstyle{Edge}=[thick];
\node[Vertex] (a) at (0,0) {\scriptsize$\mathbf1$};
\node[Vertex,inner sep = 1.25pt] (b) at (-0.5,1) {\scriptsize$\mathbf{11}$};
\node[Vertex] (c) at (-1,2) {\scriptsize$\mathbf3$};
\node[Vertex] (d) at (-1.5,3) {\scriptsize$\mathbf4$};
\node[Vertex] (e) at (0,2) {\scriptsize$\mathbf7$};
\node[Vertex] (f) at (-.5,3) {\scriptsize$\mathbf5$};
\node[Vertex] (g) at (0.5,1) {\scriptsize$\mathbf6$};
\node[Vertex,inner sep = 1.25pt] (h) at (1,2) {\scriptsize$\mathbf{10}$};
\node[Vertex] (i) at (0.5,3) {\scriptsize$\mathbf8$};
\node[Vertex] (j) at (0,4) {\scriptsize$\mathbf9$};
\node[Vertex,inner sep = 1.25pt] (k) at (3,6) {\scriptsize$\mathbf{16}$};
\node[Vertex] (l) at (1.5,3) {\scriptsize$\mathbf2$};
\node[Vertex,inner sep = 1.25pt] (m) at (2,4) {\scriptsize$\mathbf{12}$};
\node[Vertex,inner sep = 1.25pt] (n) at (2.5,5) {\scriptsize$\mathbf{15}$};
\node[Vertex,inner sep = 1.25pt] (o) at (1.5,5) {\scriptsize$\mathbf{13}$};
\node[Vertex,inner sep = 1.25pt] (p) at (1,6) {\scriptsize$\mathbf{14}$};

\draw[Edge] (a)--(b)--(c)--(d);
\draw[Edge] (c)--(f);
\draw[Edge] (a)--(g)--(e)--(i)--(j);
\draw[Edge] (g)--(h)--(l)--(m)--(o)--(p);
\draw[Edge] (m)--(n)--(k);
\end{tikzpicture}
}

\def\PICDA{
\begin{tikzpicture}[scale = 0.85]
\tikzstyle{Vertex}=[inner sep=2.5pt,thick,shape=circle,draw=black,fill=white];
\tikzstyle{Edge}=[thick];
\node[Vertex,inner sep = 1.25pt] (a) at (0,0) {\scriptsize$\mathbf{10}$};
\node[Vertex] (b) at (1,0) {\scriptsize$\mathbf3$};
\node[Vertex] (c) at (2,0) {\scriptsize$\mathbf1$};
\node[Vertex] (d) at (3,0) {\scriptsize$\mathbf2$};
\node[Vertex] (e) at (0.5,0.866) {\scriptsize$\mathbf5$};
\node[Vertex] (f) at (1.5,0.866) {\scriptsize$\mathbf4$};
\node[Vertex] (g) at (2.5,0.866) {\scriptsize$\mathbf6$};
\node[Vertex] (h) at (3.5,0.866) {\scriptsize$\mathbf8$};
\node[Vertex] (i) at (2.5,1.732) {\scriptsize$\mathbf9$};
\node[Vertex] (j) at (1.5,-0.866) {\scriptsize$\mathbf7$};

\draw[Edge] (a)--(b)--(c)--(d)--(h)--(g)--(f)--(e)--(a);
\draw[Edge] (e)--(b)--(f)--(c)--(g)--(d)--(h);
\draw[Edge] (g)--(i);
\draw[Edge] (j)--(a);
\draw[Edge] (j)--(b);
\draw[Edge] (j)--(f);
\end{tikzpicture}
}

\def\PICDB{
\begin{tikzpicture}[scale = 0.85]
\tikzstyle{Vertex}=[inner sep=2.5pt,thick,shape=circle,draw=black,fill=white];
\tikzstyle{Edge}=[thick];
\node[Vertex,inner sep = 1.25pt] (a) at (0,0) {\scriptsize$\mathbf{10}$};
\node[Vertex] (b) at (1,0) {\scriptsize$\mathbf3$};
\node[Vertex] (c) at (2,0) {\scriptsize$\mathbf1$};
\node[Vertex] (d) at (3,0) {\scriptsize$\mathbf2$};
\node[Vertex] (e) at (0.5,0.866) {\scriptsize$\mathbf5$};
\node[Vertex] (f) at (1.5,0.866) {\scriptsize$\mathbf4$};
\node[Vertex] (g) at (2.5,0.866) {\scriptsize$\mathbf6$};
\node[Vertex] (h) at (3.5,0.866) {\scriptsize$\mathbf8$};
\node[Vertex] (i) at (2.5,1.732) {\scriptsize$\mathbf9$};
\node[Vertex] (j) at (1.5,-0.866) {\scriptsize$\mathbf7$};

\draw[Edge] (a)--(e)--(b)--(f)--(c)--(g)--(d)--(c)--(b)--(a);

\draw[Edge] (g)--(b)--(j)--(a);
\draw[Edge] (d)--(h)--(i)--(g);
\draw[Edge] (i)--(f);
\end{tikzpicture}
}

\def\PICSEVENA{
\begin{tikzpicture}[scale = 1]
\tikzstyle{Vertex}=[inner sep=1pt,thick,shape=circle,draw=black,fill=white];
\tikzstyle{Edge}=[thick];

\node[Vertex] (a) at (0,2) {\scriptsize$\mathbf1$};
\node[Vertex] (b) at ({2/3},{4/3}) {\scriptsize$\mathbf3$};
\node[Vertex] (c) at ({1/3},{2/3}) {\scriptsize$\mathbf5$};
\node[Vertex] (d) at ({2/3},0) {\scriptsize$\mathbf7$};
\node[Vertex] (e) at ({-2/3},0) {\scriptsize$\mathbf6$};
\node[Vertex] (f) at ({-1/3},{2/3}) {\scriptsize$\mathbf4$};
\node[Vertex] (g) at ({-2/3},{4/3}) {\scriptsize$\mathbf2$};

\draw[Edge] (c)--(d)--(f)--(e)--(c)--(a)--(b)--(d)--(e)--(g)--(c)--(b)--(g)--(a)--(f)--(c);

\end{tikzpicture}  \hspace{1em} 
\begin{tikzpicture}[scale = 1]
\tikzstyle{Vertex}=[inner sep=1pt,thick,shape=circle,draw=black,fill=white];
\tikzstyle{Edge}=[thick];

\node[Vertex] (a) at (-{2/3},{4/3}) {\scriptsize$\mathbf2$};
\node[Vertex] (b) at (0,2) {\scriptsize$\mathbf1$};
\node[Vertex] (c) at ({-2/3},0) {\scriptsize$\mathbf7$};
\node[Vertex] (d) at ({2/3},0) {\scriptsize$\mathbf6$};
\node[Vertex] (e) at (0,1) {\scriptsize$\mathbf5$};
\node[Vertex] (f) at (0,{1/3}) {\scriptsize$\mathbf4$};
\node[Vertex] (g) at ({2/3},{4/3}) {\scriptsize$\mathbf3$};

\draw[Edge] (c)--(d)--(f)--(e)--(c)--(a)--(b)--(d)--(e);
\draw[Edge] (c)--(b)--(g)--(a)--(f)--(c);
\draw[Edge] (g)--(d);
\draw[Edge] (g)--(f);

\end{tikzpicture} & 
\begin{tikzpicture}[scale = 1]
\tikzstyle{Vertex}=[inner sep=1pt,thick,shape=circle,draw=black,fill=white];
\tikzstyle{Edge}=[thick];

\node[Vertex] (a) at ({-2/3},2) {\scriptsize$\mathbf1$};
\node[Vertex] (b) at ({2/3},2) {\scriptsize$\mathbf3$};
\node[Vertex] (c) at ({-2/3},0) {\scriptsize$\mathbf6$};
\node[Vertex] (d) at (0,{11/10}) {\scriptsize$\mathbf5$};
\node[Vertex] (e) at ({2/3},0) {\scriptsize$\mathbf4$};
\node[Vertex] (f) at (0,2) {\scriptsize$\mathbf7$};
\node[Vertex] (g) at (0,{1/2}) {\scriptsize$\mathbf2$};

\draw[Edge] (c)--(d)--(f)--(e)--(c)--(a)--(d)--(b)--(e)--(d);
\draw[Edge] (f)--(c)--(g)--(e);

\end{tikzpicture}  \hspace{1em} 
\begin{tikzpicture}[scale = 1]
\tikzstyle{Vertex}=[inner sep=1pt,thick,shape=circle,draw=black,fill=white];
\tikzstyle{Edge}=[thick];

\node[Vertex] (a) at ({2/3},2) {\scriptsize$\mathbf2$};
\node[Vertex] (b) at ({-3/5},0) {\scriptsize$\mathbf1$};
\node[Vertex] (c) at ({2/3},{2/3}) {\scriptsize$\mathbf5$};
\node[Vertex] (d) at (0,{4/3}) {\scriptsize$\mathbf7$};
\node[Vertex] (e) at (0,0) {\scriptsize$\mathbf4$};
\node[Vertex] (f) at ({-2/3},{2/3}) {\scriptsize$\mathbf6$};
\node[Vertex] (g) at ({-2/3},2) {\scriptsize$\mathbf3$};

\draw[Edge] (c)--(d)--(f)--(e)--(c)--(a)--(d)--(b)--(e)--(d);
\draw[Edge] (f)--(c);
\draw[Edge] (d)--(g)--(f);

\end{tikzpicture} & 
\begin{tikzpicture}[scale = 1]
\tikzstyle{Vertex}=[inner sep=1pt,thick,shape=circle,draw=black,fill=white];
\tikzstyle{Edge}=[thick];

\node[Vertex] (a) at ({2/3},0) {\scriptsize$\mathbf1$};
\node[Vertex] (b) at (0,2) {\scriptsize$\mathbf2$};
\node[Vertex] (c) at (0,{4/3}) {\scriptsize$\mathbf4$};
\node[Vertex] (d) at ({-2/3},{4/3}) {\scriptsize$\mathbf5$};
\node[Vertex] (e) at (0,{2/3}) {\scriptsize$\mathbf7$};
\node[Vertex] (f) at ({-2/3},0) {\scriptsize$\mathbf6$};
\node[Vertex] (g) at ({2/3},{4/3}) {\scriptsize$\mathbf3$};

\draw[Edge] (c)--(d)--(f)--(e)--(c)--(a)--(g)--(b)--(c)--(f)--(e)--(d)--(b);
\draw[Edge] (f)--(a);
\draw[Edge] (c)--(g)--(e);

\end{tikzpicture}  \hspace{1em} 
\begin{tikzpicture}[scale = 1]
\tikzstyle{Vertex}=[inner sep=1pt,thick,shape=circle,draw=black,fill=white];
\tikzstyle{Edge}=[thick];

\node[Vertex] (a) at ({2/3},{3/2}) {\scriptsize$\mathbf1$};
\node[Vertex] (b) at ({-2/3},{3/2}) {\scriptsize$\mathbf2$};
\node[Vertex] (c) at (0,1) {\scriptsize$\mathbf7$};
\node[Vertex] (d) at ({-2/3},{1/2}) {\scriptsize$\mathbf6$};
\node[Vertex] (e) at (0,0) {\scriptsize$\mathbf4$};
\node[Vertex] (f) at ({2/3},{1/2}) {\scriptsize$\mathbf5$};
\node[Vertex] (g) at (0,2) {\scriptsize$\mathbf3$};

\draw[Edge] (c)--(d)--(f)--(e)--(c)--(a)--(g)--(b)--(c)--(f)--(e)--(d)--(b);
\draw[Edge] (f)--(a);
\draw[Edge] (d)--(g)--(f);

\end{tikzpicture}
}

\def\PICSEVENB{
\begin{tikzpicture}[scale = 1]
\tikzstyle{Vertex}=[inner sep=1pt,thick,shape=circle,draw=black,fill=white];
\tikzstyle{Edge}=[thick];

\node[Vertex] (a) at ({-2/3},0) {\scriptsize$\mathbf3$};
\node[Vertex] (b) at (0,2) {\scriptsize$\mathbf1$};
\node[Vertex] (c) at (0,{1/2}) {\scriptsize$\mathbf5$};
\node[Vertex] (d) at ({-2/3},1) {\scriptsize$\mathbf6$};
\node[Vertex] (e) at ({2/3},1) {\scriptsize$\mathbf4$};
\node[Vertex] (f) at (0,1.4) {\scriptsize$\mathbf7$};
\node[Vertex] (g) at ({2/3},0) {\scriptsize$\mathbf2$};

\draw[Edge] (c)--(d)--(f)--(e)--(c)--(a)--(d)--(e)--(b)--(d)--(a)--(g);
\draw[Edge] (c)--(f);
\draw[Edge] (c)--(g)--(e);

\end{tikzpicture}  \hspace{1em}
\begin{tikzpicture}[scale = 1]
\tikzstyle{Vertex}=[inner sep=1pt,thick,shape=circle,draw=black,fill=white];
\tikzstyle{Edge}=[thick];

\node[Vertex] (a) at ({2/3},0) {\scriptsize$\mathbf2$};
\node[Vertex] (b) at (0,{2}) {\scriptsize$\mathbf1$};
\node[Vertex] (c) at ({-2/3},0) {\scriptsize$\mathbf6$};
\node[Vertex] (d) at (0,{4/3}) {\scriptsize$\mathbf7$};
\node[Vertex] (e) at ({-2/3},{4/3}) {\scriptsize$\mathbf5$};
\node[Vertex] (f) at (0,{1/2}) {\scriptsize$\mathbf4$};
\node[Vertex] (g) at ({2/3},{4/3}) {\scriptsize$\mathbf3$};

\draw[Edge] (c)--(d)--(f)--(e)--(c)--(a)--(d)--(e)--(b)--(d)--(a)--(g);
\draw[Edge] (c)--(f);
\draw[Edge] (d)--(g)--(f);

\end{tikzpicture} & 
\begin{tikzpicture}[scale = 1]
\tikzstyle{Vertex}=[inner sep=1pt,thick,shape=circle,draw=black,fill=white];
\tikzstyle{Edge}=[thick];

\node[Vertex] (a) at ({-2/3},0) {\scriptsize$\mathbf1$};
\node[Vertex] (b) at ({2/3},0) {\scriptsize$\mathbf3$};
\node[Vertex] (c) at (0,{1/2}) {\scriptsize$\mathbf4$};
\node[Vertex] (d) at ({2/3},1) {\scriptsize$\mathbf6$};
\node[Vertex] (e) at ({-2/3},2) {\scriptsize$\mathbf5$};
\node[Vertex] (f) at ({-2/3},1) {\scriptsize$\mathbf7$};
\node[Vertex] (g) at ({2/3},2) {\scriptsize$\mathbf2$};

\draw[Edge] (c)--(d)--(f)--(e)--(c)--(a)--(b)--(d);
\draw[Edge] (b)--(c);
\draw[Edge] (f)--(a);
\draw[Edge] (c)--(g)--(e);

\end{tikzpicture}  \hspace{1em} 
\begin{tikzpicture}[scale = 1]
\tikzstyle{Vertex}=[inner sep=1pt,thick,shape=circle,draw=black,fill=white];
\tikzstyle{Edge}=[thick];

\node[Vertex] (a) at ({-2/3},0) {\scriptsize$\mathbf1$};
\node[Vertex] (b) at ({2/3},0) {\scriptsize$\mathbf3$};
\node[Vertex] (c) at (0,{1/2}) {\scriptsize$\mathbf5$};
\node[Vertex] (d) at ({2/3},1) {\scriptsize$\mathbf7$};
\node[Vertex] (e) at ({-2/3},2) {\scriptsize$\mathbf4$};
\node[Vertex] (f) at ({-2/3},1) {\scriptsize$\mathbf6$};
\node[Vertex] (g) at ({2/3},2) {\scriptsize$\mathbf2$};

\draw[Edge] (c)--(d)--(f)--(e)--(c)--(a)--(b)--(d);
\draw[Edge] (b)--(c);
\draw[Edge] (f)--(a);
\draw[Edge] (d)--(g)--(f);

\end{tikzpicture} & 
\begin{tikzpicture}[scale = 1]
\tikzstyle{Vertex}=[inner sep=1pt,thick,shape=circle,draw=black,fill=white];
\tikzstyle{Edge}=[thick];

\node[Vertex] (a) at ({2/3},1) {\scriptsize$\mathbf7$};
\node[Vertex] (b) at ({-2/3},0) {\scriptsize$\mathbf2$};
\node[Vertex] (c) at ({-2/3},1) {\scriptsize$\mathbf6$};
\node[Vertex] (d) at (0,0) {\scriptsize$\mathbf4$};
\node[Vertex] (e) at (0,2) {\scriptsize$\mathbf5$};
\node[Vertex] (f) at (0,1) {\scriptsize$\mathbf1$};
\node[Vertex] (g) at ({-2/3},2) {\scriptsize$\mathbf3$};

\draw[Edge] (c)--(d)--(f)--(e)--(c)--(b)--(d)--(a);
\draw[Edge] (c)--(g)--(e);
\draw[Edge] (a)--(e);

\end{tikzpicture} \hspace{1em} 
\begin{tikzpicture}[scale = 1]
\tikzstyle{Vertex}=[inner sep=1pt,thick,shape=circle,draw=black,fill=white];
\tikzstyle{Edge}=[thick];

\node[Vertex] (a) at (0,1) {\scriptsize$\mathbf1$};
\node[Vertex] (b) at ({2/3},2) {\scriptsize$\mathbf3$};
\node[Vertex] (c) at ({2/3},1) {\scriptsize$\mathbf4$};
\node[Vertex] (d) at (0,2) {\scriptsize$\mathbf7$};
\node[Vertex] (e) at (0,0) {\scriptsize$\mathbf6$};
\node[Vertex] (f) at ({-2/3},1) {\scriptsize$\mathbf5$};
\node[Vertex] (g) at ({-2/3},2) {\scriptsize$\mathbf2$};

\draw[Edge] (c)--(d)--(f)--(e)--(c)--(b)--(d)--(a) (d)--(g)--(f)  (a)--(e);

\end{tikzpicture} 
}

\def\PICSEVENC{
\begin{tikzpicture}[scale = 1]
\tikzstyle{Vertex}=[inner sep=1pt,thick,shape=circle,draw=black,fill=white];
\tikzstyle{Edge}=[thick];

\node[Vertex] (a) at ({1/3},{2/3}) {\scriptsize$\mathbf3$};
\node[Vertex] (b) at ({-2/3},{4/3}) {\scriptsize$\mathbf1$};
\node[Vertex] (c) at (0,{4/3}) {\scriptsize$\mathbf4$};
\node[Vertex] (d) at ({-2/3},0) {\scriptsize$\mathbf6$};
\node[Vertex] (e) at ({2/3},2) {\scriptsize$\mathbf5$};
\node[Vertex] (f) at ({2/3},0) {\scriptsize$\mathbf7$};
\node[Vertex] (g) at ({-2/3},2) {\scriptsize$\mathbf2$};

\draw[Edge] (c)--(d)--(f)--(e)--(c)--(a)--(f);
\draw[Edge] (c)--(b)--(d)--(b)--(g);
\draw[Edge] (c)--(g)--(e);

\end{tikzpicture}  \hspace{1em} 
\begin{tikzpicture}[scale = 1]
\tikzstyle{Vertex}=[inner sep=1pt,thick,shape=circle,draw=black,fill=white];
\tikzstyle{Edge}=[thick];

\node[Vertex] (a) at ({2/3},2) {\scriptsize$\mathbf3$};
\node[Vertex] (b) at ({-2/3},0) {\scriptsize$\mathbf1$};
\node[Vertex] (c) at ({-2/3},1) {\scriptsize$\mathbf5$};
\node[Vertex] (d) at (0,{1/2}) {\scriptsize$\mathbf7$};
\node[Vertex] (e) at ({-2/3},2) {\scriptsize$\mathbf4$};
\node[Vertex] (f) at ({2/3},1) {\scriptsize$\mathbf6$};
\node[Vertex] (g) at ({2/3},0) {\scriptsize$\mathbf2$};

\draw[Edge] (c)--(d)--(f)--(e)--(c)--(a)--(f);
\draw[Edge] (c)--(b)--(d)--(b)--(g);
\draw[Edge] (d)--(g)--(f);

\end{tikzpicture} & 
\begin{tikzpicture}[scale = 1]
\tikzstyle{Vertex}=[inner sep=1pt,thick,shape=circle,draw=black,fill=white];
\tikzstyle{Edge}=[thick];

\node[Vertex] (a) at (0,{5/4}) {\scriptsize$\mathbf1$};
\node[Vertex] (b) at ({-2/3},0) {\scriptsize$\mathbf2$};
\node[Vertex] (c) at (0,{1/2}) {\scriptsize$\mathbf6$};
\node[Vertex] (d) at ({-2/3},{5/4}) {\scriptsize$\mathbf4$};
\node[Vertex] (e) at ({2/3},{5/4}) {\scriptsize$\mathbf5$};
\node[Vertex] (f) at (0,2) {\scriptsize$\mathbf7$};
\node[Vertex] (g) at ({2/3},0) {\scriptsize$\mathbf3$};

\draw[Edge] (c)--(d)--(f)--(e)--(c)--(a)--(b)--(c);
\draw[Edge] (d)--(b)--(g)--(a)--(f);
\draw[Edge] (c)--(g)--(e);

\end{tikzpicture}  \hspace{1em} 
\begin{tikzpicture}[scale = 1]
\tikzstyle{Vertex}=[inner sep=1pt,thick,shape=circle,draw=black,fill=white];
\tikzstyle{Edge}=[thick];

\node[Vertex] (a) at ({-1/3},{5/8}) {\scriptsize$\mathbf1$};
\node[Vertex] (b) at ({-2/3},0) {\scriptsize$\mathbf2$};
\node[Vertex] (c) at ({-2/3},{5/4}) {\scriptsize$\mathbf5$};
\node[Vertex] (d) at ({1/3},{5/8}) {\scriptsize$\mathbf7$};
\node[Vertex] (e) at (0,2) {\scriptsize$\mathbf6$};
\node[Vertex] (f) at ({2/3},{5/4}) {\scriptsize$\mathbf4$};
\node[Vertex] (g) at ({2/3},0) {\scriptsize$\mathbf3$};

\draw[Edge] (c)--(d)--(f)--(e)--(c)--(a)--(b)--(c);
\draw[Edge] (d)--(b)--(g)--(a)--(f);
\draw[Edge] (d)--(g)--(f);

\end{tikzpicture} & 
\begin{tikzpicture}[scale = 1]
\tikzstyle{Vertex}=[inner sep=1pt,thick,shape=circle,draw=black,fill=white];
\tikzstyle{Edge}=[thick];

\node[Vertex] (a) at ({2/3},2) {\scriptsize$\mathbf2$};
\node[Vertex] (b) at (0,0.75) {\scriptsize$\mathbf3$};
\node[Vertex] (c) at (0,1.5) {\scriptsize$\mathbf4$};
\node[Vertex] (d) at ({2/3},0.75) {\scriptsize$\mathbf5$};
\node[Vertex] (e) at (-{2/3},0.75) {\scriptsize$\mathbf6$};
\node[Vertex] (f) at (0,0) {\scriptsize$\mathbf7$};
\node[Vertex] (g) at (-{2/3},2) {\scriptsize$\mathbf1$};

\draw[Edge] (c)--(d)--(f)--(e)--(c)--(a)--(b)--(c);
\draw[Edge] (a)--(d);
\draw[Edge] (g)--(e);
\draw[Edge] (g)--(c);
\draw[Edge] (g)--(b);
\draw[Edge] (b)--(f);

\end{tikzpicture}  \hspace{1em}
\begin{tikzpicture}[scale = 1]
\tikzstyle{Vertex}=[inner sep=1pt,thick,shape=circle,draw=black,fill=white];
\tikzstyle{Edge}=[thick];

\node[Vertex] (a) at ({2/3},1.5) {\scriptsize$\mathbf1$};
\node[Vertex] (b) at (0,2) {\scriptsize$\mathbf7$};
\node[Vertex] (c) at ({2/3},0.5) {\scriptsize$\mathbf5$};
\node[Vertex] (d) at (0,1) {\scriptsize$\mathbf3$};
\node[Vertex] (e) at (0,0) {\scriptsize$\mathbf4$};
\node[Vertex] (f) at (-{2/3},0.5) {\scriptsize$\mathbf6$};
\node[Vertex] (g) at (-{2/3},1.5) {\scriptsize$\mathbf2$};

\draw[Edge] (c)--(d)--(f)--(e)--(c)--(a)--(b)--(c);
\draw[Edge] (a)--(d);
\draw[Edge] (g)--(f);
\draw[Edge] (g)--(d);
\draw[Edge] (g)--(b);
\draw[Edge] (b)--(f);

\end{tikzpicture}
}

\def\PICSEVEND{
\begin{tikzpicture}[scale = 1]
\tikzstyle{Vertex}=[inner sep=1pt,thick,shape=circle,draw=black,fill=white];
\tikzstyle{Edge}=[thick];

\node[Vertex] (a) at (0,2) {\scriptsize$\mathbf1$};
\node[Vertex] (b) at ({2/3},1.25) {\scriptsize$\mathbf2$};
\node[Vertex] (c) at (0,1.25) {\scriptsize$\mathbf4$};
\node[Vertex] (d) at ({2/3},0) {\scriptsize$\mathbf6$};
\node[Vertex] (e) at (-{2/3},1.25) {\scriptsize$\mathbf5$};
\node[Vertex] (f) at (-{2/3},0) {\scriptsize$\mathbf7$};
\node[Vertex] (g) at (0,0.625) {\scriptsize$\mathbf3$};

\draw[Edge] (c)--(d)--(f)--(e)--(c)--(a)--(b)--(c)--(g)--(b)--(d);
\draw[Edge] (f)--(g);
\draw[Edge] (a)--(e);

\end{tikzpicture}  \hspace{1em} 
\begin{tikzpicture}[scale = 1]
\tikzstyle{Vertex}=[inner sep=1pt,thick,shape=circle,draw=black,fill=white];
\tikzstyle{Edge}=[thick];

\node[Vertex] (a) at (-{2/3},0) {\scriptsize$\mathbf1$};
\node[Vertex] (b) at ({2/3},0) {\scriptsize$\mathbf2$};
\node[Vertex] (c) at (0,0.5) {\scriptsize$\mathbf7$};
\node[Vertex] (d) at ({2/3},1.5) {\scriptsize$\mathbf5$};
\node[Vertex] (e) at (-{2/3},1.5) {\scriptsize$\mathbf6$};
\node[Vertex] (f) at (0,2) {\scriptsize$\mathbf4$};
\node[Vertex] (g) at (0,1) {\scriptsize$\mathbf3$};

\draw[Edge] (c)--(d)--(f)--(e)--(c)--(a)--(b)--(c);
\draw[Edge] (g)--(b)--(d);
\draw[Edge] (d)--(g);
\draw[Edge] (d)--(e)--(g);
\draw[Edge] (a)--(e);

\end{tikzpicture} & 
\begin{tikzpicture}[scale = 1]
\tikzstyle{Vertex}=[inner sep=1pt,thick,shape=circle,draw=black,fill=white];
\tikzstyle{Edge}=[thick];

\node[Vertex] (a) at (0,2) {\scriptsize$\mathbf3$};
\node[Vertex] (b) at ({2/3},1.5) {\scriptsize$\mathbf1$};
\node[Vertex] (c) at (0,1) {\scriptsize$\mathbf5$};
\node[Vertex] (d) at ({2/3},0.5) {\scriptsize$\mathbf2$};
\node[Vertex] (e) at (0,0) {\scriptsize$\mathbf7$};
\node[Vertex] (f) at (-{2/3},0.5) {\scriptsize$\mathbf4$};
\node[Vertex] (g) at (-{2/3},1.5) {\scriptsize$\mathbf6$};

\draw[Edge] (c)--(d)--(f)--(e)--(c)--(a)--(b)--(a)--(c)--(g)--(b)--(d);
\draw[Edge] (c)--(b);
\draw[Edge] (g)--(f);

\end{tikzpicture}  \hspace{1em} 
\begin{tikzpicture}[scale = 1]
\tikzstyle{Vertex}=[inner sep=1pt,thick,shape=circle,draw=black,fill=white];
\tikzstyle{Edge}=[thick];

\node[Vertex] (a) at (0,2) {\scriptsize$\mathbf3$};
\node[Vertex] (b) at ({2/3},1.5) {\scriptsize$\mathbf1$};
\node[Vertex] (c) at (-{2/3},1.5) {\scriptsize$\mathbf4$};
\node[Vertex] (d) at (0,1) {\scriptsize$\mathbf7$};
\node[Vertex] (e) at (-{2/3},0.5) {\scriptsize$\mathbf6$};
\node[Vertex] (f) at (0,0) {\scriptsize$\mathbf5$};
\node[Vertex] (g) at ({2/3},0.5) {\scriptsize$\mathbf2$};

\draw[Edge] (c)--(d)--(f)--(e)--(c)--(a)--(b)--(a);
\draw[Edge] (d)--(g)--(b)--(d);
\draw[Edge] (c)--(b);
\draw[Edge] (g)--(e);

\end{tikzpicture} & 
}

\def\PICNINEA{
\begin{tikzpicture}[yscale = 0.4,xscale=0.36]
\tikzstyle{Vertex}=[inner sep=1.5pt,thick,shape=circle,draw=black,fill=white];
\tikzstyle{Edge}=[thick];

\node[Vertex] (8) at (90:1.2) {};
\node[Vertex] (6) at (210:1.2) {};
\node[Vertex] (7) at (-30:1.2) {};
\node[Vertex] (0) at (-150:3.5) {};
\node[Vertex] (1) at (-30:3.5) {};
\node[Vertex] (3) at (30:3.5) {};
\node[Vertex] (2) at (150:3.5) {};
\node[Vertex] (4) at (90:{3.5}) {};
\node[Vertex] (5) at (-90:{3.5}) {};

\draw[Edge](0)--(5);
\draw[Edge](0)--(6);
\draw[Edge](0)--(7);
\draw[Edge](0)--(8);
\draw[Edge](1)--(5);
\draw[Edge](1)--(6);
\draw[Edge](1)--(7);
\draw[Edge](1)--(8);
\draw[Edge](2)--(6);
\draw[Edge](2)--(7);
\draw[Edge](2)--(8);
\draw[Edge](3)--(6);
\draw[Edge](3)--(7);
\draw[Edge](3)--(8);
\draw[Edge](4)--(6);
\draw[Edge](4)--(7);
\draw[Edge](4)--(8);
\draw[Edge](5)--(6);
\draw[Edge](5)--(7);
\draw[Edge](5)--(8);
\draw[Edge](6)--(8);
\draw[Edge](7)--(8);

\end{tikzpicture}
\begin{tikzpicture}[yscale = 0.65,xscale=0.585]
\tikzstyle{Vertex}=[inner sep=1.5pt,thick,shape=circle,draw=black,fill=white];
\tikzstyle{Edge}=[thick];

\node[Vertex] (0) at (210:2) {};
\node[Vertex] (1) at (330:2) {};
\node[Vertex] (2) at (150:2) {};
\node[Vertex] (3) at (30:2) {};
\node[Vertex] (4) at (0, 3) {};
\node[Vertex] (5) at (0.75,0) {};
\node[Vertex] (6) at (-0.75,0) {};
\node[Vertex] (7) at (90:2) {};
\node[Vertex] (8) at (270:2) {};

\draw[Edge](0)--(5);
\draw[Edge](0)--(6);
\draw[Edge](0)--(7);
\draw[Edge](0)--(8);
\draw[Edge](1)--(5);
\draw[Edge](1)--(6);
\draw[Edge](1)--(7);
\draw[Edge](1)--(8);
\draw[Edge](2)--(5);
\draw[Edge](2)--(6);
\draw[Edge](2)--(7);
\draw[Edge](2)--(8);
\draw[Edge](3)--(5);
\draw[Edge](3)--(6);
\draw[Edge](3)--(7);
\draw[Edge](3)--(8);
\draw[Edge](5)--(7);
\draw[Edge](5)--(8);
\draw[Edge](6)--(7);
\draw[Edge](6)--(8);
\draw[Edge](7)--(8);
\draw[Edge](4)--(7);

\end{tikzpicture} &
\begin{tikzpicture}[yscale = 0.65,xscale=0.585]
\tikzstyle{Vertex}=[inner sep=1.5pt,thick,shape=circle,draw=black,fill=white];
\tikzstyle{Edge}=[thick];

\node[Vertex] (0) at ({2*cos(234)},{2 +0.5*sqrt(2)}) {};
\node[Vertex] (1) at (198:0.4) {};
\node[Vertex] (2) at (-18:0.4) {};
\node[Vertex] (3) at ({2*cos(306)},{2 +0.5*sqrt(2)}) {};
\node[Vertex] (4) at (234: 2) {};
\node[Vertex] (5) at (306:2) {};
\node[Vertex] (6) at (18:2) {};
\node[Vertex] (8) at (90:2) {};
\node[Vertex] (7) at (162:2) {};

\draw[Edge](0)--(3);
\draw[Edge](0)--(8);
\draw[Edge](1)--(4);
\draw[Edge](1)--(5);
\draw[Edge](1)--(6);
\draw[Edge](1)--(7);
\draw[Edge](1)--(8);
\draw[Edge](2)--(4);
\draw[Edge](2)--(5);
\draw[Edge](2)--(6);
\draw[Edge](2)--(7);
\draw[Edge](2)--(8);
\draw[Edge](3)--(8);
\draw[Edge](4)--(5);
\draw[Edge](4)--(6);
\draw[Edge](4)--(7);
\draw[Edge](4)--(8);
\draw[Edge](5)--(6);
\draw[Edge](5)--(7);
\draw[Edge](5)--(8);
\draw[Edge](6)--(7);
\draw[Edge](6)--(8);
\draw[Edge](7)--(8);

\end{tikzpicture}
\begin{tikzpicture}[yscale = 0.4,xscale=0.36]
\tikzstyle{Vertex}=[inner sep=1.5pt,thick,shape=circle,draw=black,fill=white];
\tikzstyle{Edge}=[thick];

\node[Vertex] (0) at ({-2*cos(306)}, {2*sin(306) -2*sqrt(3)}) {};
\node[Vertex] (1) at (18:2) {};
\node[Vertex] (2) at (162:2) {};
\node[Vertex] (3) at ({2*cos(306) + 1},{2*sin(306) - sqrt(3)}) {};
\node[Vertex] (4) at ({2*cos(306)}, {2*sin(306) -2*sqrt(3)}) {};
\node[Vertex] (5) at (90:2) {};
\node[Vertex] (6) at ({-2*cos(306) - 1},{2*sin(306) - sqrt(3)}) {};
\node[Vertex] (8) at (304:2) {};
\node[Vertex] (7) at (234:2) {};

\draw[Edge](0)--(3);
\draw[Edge](0)--(4);
\draw[Edge](0)--(6);
\draw[Edge](0)--(7);
\draw[Edge](0)--(8);
\draw[Edge](1)--(5);
\draw[Edge](1)--(7);
\draw[Edge](1)--(8);
\draw[Edge](2)--(5);
\draw[Edge](2)--(7);
\draw[Edge](2)--(8);
\draw[Edge](3)--(4);
\draw[Edge](3)--(6);
\draw[Edge](3)--(7);
\draw[Edge](3)--(8);
\draw[Edge](4)--(6);
\draw[Edge](4)--(7);
\draw[Edge](4)--(8);
\draw[Edge](5)--(7);
\draw[Edge](5)--(8);
\draw[Edge](6)--(7);
\draw[Edge](6)--(8);
\draw[Edge](7)--(8);

\end{tikzpicture} &
\begin{tikzpicture}[yscale = 0.5,xscale=0.45]
\tikzstyle{Vertex}=[inner sep=1.5pt,thick,shape=circle,draw=black,fill=white];
\tikzstyle{Edge}=[thick];

\node[Vertex] (0) at ({2*cos(306)+1},{2*sin(306) -1}) {};
\node[Vertex] (1) at ({2*cos(234) - 1},{2*sin(234) -1}) {};
\node[Vertex] (2) at (162:2) {};
\node[Vertex] (3) at (18:2) {};
\node[Vertex] (4) at ({2*cos(306)} , {2*sin(306)-2}) {};
\node[Vertex] (5) at ({2*cos(234)},{2*sin(234) - 2}) {};
\node[Vertex] (6) at (90:2) {};
\node[Vertex] (8) at (234:2) {};
\node[Vertex] (7) at (304:2) {};

\draw[Edge](0)--(4);
\draw[Edge](0)--(7);
\draw[Edge](0)--(8);
\draw[Edge](1)--(5);
\draw[Edge](1)--(7);
\draw[Edge](1)--(8);
\draw[Edge](2)--(6);
\draw[Edge](2)--(7);
\draw[Edge](2)--(8);
\draw[Edge](3)--(6);
\draw[Edge](3)--(7);
\draw[Edge](3)--(8);
\draw[Edge](4)--(7);
\draw[Edge](4)--(8);
\draw[Edge](5)--(7);
\draw[Edge](5)--(8);
\draw[Edge](6)--(7);
\draw[Edge](6)--(8);
\draw[Edge](7)--(8);

\end{tikzpicture}
\begin{tikzpicture}[yscale = 0.6,xscale=0.54]
\tikzstyle{Vertex}=[inner sep=1.5pt,thick,shape=circle,draw=black,fill=white];
\tikzstyle{Edge}=[thick];

\node[Vertex] (0) at (60:2) {};
\node[Vertex] (1) at (240:2) {};
\node[Vertex] (2) at ({2*cos(120)},{2*sin(120) + 1.2}) {};
\node[Vertex] (3) at ({2*cos(60)},{sqrt(3) + 1.2}) {};
\node[Vertex] (4) at (120:2) {};
\node[Vertex] (5) at (300:2) {};
\node[Vertex] (6) at (180:2) {};
\node[Vertex] (8) at (0,{sqrt(3)+0.6}) {};
\node[Vertex] (7) at (0:2) {};

\draw[Edge](0)--(4);
\draw[Edge](0)--(6);
\draw[Edge](0)--(7);
\draw[Edge](0)--(8);
\draw[Edge](1)--(5);
\draw[Edge](1)--(6);
\draw[Edge](1)--(7);
\draw[Edge](1)--(8);
\draw[Edge](2)--(8);
\draw[Edge](3)--(8);
\draw[Edge](4)--(6);
\draw[Edge](4)--(7);
\draw[Edge](4)--(8);
\draw[Edge](5)--(6);
\draw[Edge](5)--(7);
\draw[Edge](5)--(8);
\draw[Edge](6)--(7);
\draw[Edge](6)--(8);
\draw[Edge](7)--(8);
\end{tikzpicture} 
}

\def\PICNINEB{
\begin{tikzpicture}[yscale = 0.55,xscale=0.495]
\tikzstyle{Vertex}=[inner sep=1.5pt,thick,shape=circle,draw=black,fill=white];
\tikzstyle{Edge}=[thick];

\node[Vertex] (0) at ({2*cos(30)},2) {};
\node[Vertex] (1) at ({2*cos(150)},2) {};
\node[Vertex] (2) at (-150:2) {};
\node[Vertex] (3) at (-90:2) {};
\node[Vertex] (4) at (-30:2) {};
\node[Vertex] (5) at (90:3) {};
\node[Vertex] (6) at (30:2) {};
\node[Vertex] (8) at (90:2) {};
\node[Vertex] (7) at (150:2) {};

\draw[Edge](0)--(5);
\draw[Edge](0)--(6);
\draw[Edge](0)--(8);
\draw[Edge](1)--(5);
\draw[Edge](1)--(7);
\draw[Edge](1)--(8);
\draw[Edge](2)--(6);
\draw[Edge](2)--(7);
\draw[Edge](2)--(8);
\draw[Edge](3)--(6);
\draw[Edge](3)--(7);
\draw[Edge](3)--(8);
\draw[Edge](4)--(6);
\draw[Edge](4)--(7);
\draw[Edge](4)--(8);
\draw[Edge](5)--(6);
\draw[Edge](5)--(7);
\draw[Edge](5)--(8);
\draw[Edge](6)--(8);
\draw[Edge](7)--(8);
\end{tikzpicture}
\begin{tikzpicture}[yscale = 0.7,xscale=0.63]
\tikzstyle{Vertex}=[inner sep=1.5pt,thick,shape=circle,draw=black,fill=white];
\tikzstyle{Edge}=[thick];

\node[Vertex] (0) at (150:2) {};
\node[Vertex] (1) at (-30:2) {};
\node[Vertex] (2) at (-0.74,0) {};
\node[Vertex] (3) at (0.75,0) {};
\node[Vertex] (4) at (-150:2) {};
\node[Vertex] (5) at (30:2) {};
\node[Vertex] (6) at (0:0) {};
\node[Vertex] (8) at (90:2) {};
\node[Vertex] (7) at (-90:2) {};

\draw[Edge](0)--(4);
\draw[Edge](0)--(6);
\draw[Edge](0)--(7);
\draw[Edge](0)--(8);
\draw[Edge](1)--(5);
\draw[Edge](1)--(6);
\draw[Edge](1)--(7);
\draw[Edge](1)--(8);
\draw[Edge](2)--(7);
\draw[Edge](2)--(8);
\draw[Edge](3)--(7);
\draw[Edge](3)--(8);
\draw[Edge](4)--(6);
\draw[Edge](4)--(7);
\draw[Edge](4)--(8);
\draw[Edge](5)--(6);
\draw[Edge](5)--(7);
\draw[Edge](5)--(8);
\draw[Edge](6)--(7);
\draw[Edge](6)--(8);

\end{tikzpicture} &
\begin{tikzpicture}[yscale = 0.7,xscale=0.63]
\tikzstyle{Vertex}=[inner sep=1.5pt,thick,shape=circle,draw=black,fill=white];
\tikzstyle{Edge}=[thick];

\node[Vertex] (0) at (150:2) {};
\node[Vertex] (1) at (30:2) {};
\node[Vertex] (2) at (-0.75,0) {};
\node[Vertex] (3) at (0.75,0) {};
\node[Vertex] (4) at (-150:2) {};
\node[Vertex] (5) at (-30:2) {};
\node[Vertex] (6) at (90:2) {};
\node[Vertex] (8) at (0:0) {};
\node[Vertex] (7) at (-90:2) {};

\draw[Edge](0)--(4);
\draw[Edge](0)--(6);
\draw[Edge](0)--(7);
\draw[Edge](0)--(8);
\draw[Edge](1)--(5);
\draw[Edge](1)--(6);
\draw[Edge](1)--(7);
\draw[Edge](1)--(8);
\draw[Edge](2)--(6);
\draw[Edge](2)--(7);
\draw[Edge](2)--(8);
\draw[Edge](3)--(6);
\draw[Edge](3)--(7);
\draw[Edge](3)--(8);
\draw[Edge](4)--(7);
\draw[Edge](4)--(8);
\draw[Edge](5)--(7);
\draw[Edge](5)--(8);
\draw[Edge](6)--(8);
\draw[Edge](7)--(8);
\end{tikzpicture}
\begin{tikzpicture}[yscale = 0.7,xscale=0.63]
\tikzstyle{Vertex}=[inner sep=1.5pt,thick,shape=circle,draw=black,fill=white];
\tikzstyle{Edge}=[thick];

\node[Vertex] (0) at (150:2) {};
\node[Vertex] (1) at (-30:2) {};
\node[Vertex] (2) at (-0.74,0) {};
\node[Vertex] (3) at (0.75,0) {};
\node[Vertex] (4) at (-150:2) {};
\node[Vertex] (5) at (30:2) {};
\node[Vertex] (6) at (0:0) {};
\node[Vertex] (8) at (90:2) {};
\node[Vertex] (7) at (-90:2) {};

\draw[Edge](0)--(4);
\draw[Edge](0)--(6);
\draw[Edge](0)--(7);
\draw[Edge](0)--(8);
\draw[Edge](1)--(5);
\draw[Edge](1)--(6);
\draw[Edge](1)--(7);
\draw[Edge](1)--(8);
\draw[Edge](2)--(7);
\draw[Edge](2)--(8);
\draw[Edge](3)--(7);
\draw[Edge](3)--(8);
\draw[Edge](4)--(6);
\draw[Edge](4)--(7);
\draw[Edge](4)--(8);
\draw[Edge](5)--(6);
\draw[Edge](5)--(7);
\draw[Edge](5)--(8);
\draw[Edge](6)--(7);
\draw[Edge](6)--(8);

\end{tikzpicture} &
\begin{tikzpicture}[yscale = 0.7,xscale=0.63]
\tikzstyle{Vertex}=[inner sep=1.5pt,thick,shape=circle,draw=black,fill=white];
\tikzstyle{Edge}=[thick];

\node[Vertex] (2) at (210:2) {};
\node[Vertex] (1) at (150:2) {};
\node[Vertex] (0) at (30:2) {};
\node[Vertex] (3) at (270:2){};
\node[Vertex] (4) at (0,0) {};
\node[Vertex] (5) at (330:2) {};
\node[Vertex] (6) at (90:2) {};
\node[Vertex] (8) at (0.75, 0) {};
\node[Vertex] (7) at (-0.75,0) {};

\draw[Edge](0)--(6);
\draw[Edge](0)--(7);
\draw[Edge](0)--(8);
\draw[Edge](1)--(6);
\draw[Edge](1)--(7);
\draw[Edge](1)--(8);
\draw[Edge](2)--(7);
\draw[Edge](2)--(8);
\draw[Edge](3)--(7);
\draw[Edge](3)--(8);
\draw[Edge](4)--(7);
\draw[Edge](4)--(8);
\draw[Edge](5)--(7);
\draw[Edge](5)--(8);
\draw[Edge](6)--(7);
\draw[Edge](6)--(8);

\end{tikzpicture}
\begin{tikzpicture}[yscale = 0.7,xscale=0.63]
\tikzstyle{Vertex}=[inner sep=1.5pt,thick,shape=circle,draw=black,fill=white];
\tikzstyle{Edge}=[thick];

\node[Vertex] (2) at (30:2) {};
\node[Vertex] (1) at (210:2) {};
\node[Vertex] (0) at (150:2) {};
\node[Vertex] (3) at (330:2){};
\node[Vertex] (4) at (-0.75, 0) {};
\node[Vertex] (5) at (0.75, 0) {};
\node[Vertex] (6) at (90:2) {};
\node[Vertex] (8) at (0,0) {};
\node[Vertex] (7) at (270:2) {};

\draw[Edge](0)--(6);
\draw[Edge](0)--(7);
\draw[Edge](0)--(8);
\draw[Edge](1)--(6);
\draw[Edge](1)--(7);
\draw[Edge](1)--(8);
\draw[Edge](2)--(6);
\draw[Edge](2)--(7);
\draw[Edge](2)--(8);
\draw[Edge](3)--(6);
\draw[Edge](3)--(7);
\draw[Edge](3)--(8);
\draw[Edge](4)--(8);
\draw[Edge](5)--(8);
\draw[Edge](6)--(8);
\draw[Edge](7)--(8);

\end{tikzpicture} 
}

\def\PICNINEC{
\begin{tikzpicture}[yscale = 0.5,xscale=0.45]
\tikzstyle{Vertex}=[inner sep=1.5pt,thick,shape=circle,draw=black,fill=white];
\tikzstyle{Edge}=[thick];

\node[Vertex] (2) at (-18:{2*cos(36) + 0.75*sqrt(2)}) {};
\node[Vertex] (1) at (-90:{2* cos(36) + 0.5*sqrt(2)}) {};
\node[Vertex] (0) at (198:{2*cos(36)+0.75*sqrt(2)}) {};
\node[Vertex] (3) at (90:{2 + 0.75*sqrt(2)}){};
\node[Vertex] (4) at (234:2) {};
\node[Vertex] (5) at (306:2) {};
\node[Vertex] (6) at (162:2) {};
\node[Vertex] (8) at (18:2) {};
\node[Vertex] (7) at (90:2) {};

\draw[Edge](0)--(4);
\draw[Edge](0)--(5);
\draw[Edge](0)--(6);
\draw[Edge](0)--(7);
\draw[Edge](1)--(4);
\draw[Edge](1)--(5);
\draw[Edge](1)--(6);
\draw[Edge](1)--(8);
\draw[Edge](2)--(4);
\draw[Edge](2)--(5);
\draw[Edge](2)--(7);
\draw[Edge](2)--(8);
\draw[Edge](3)--(6);
\draw[Edge](3)--(7);
\draw[Edge](3)--(8);
\draw[Edge](4)--(6);
\draw[Edge](4)--(7);
\draw[Edge](4)--(8);
\draw[Edge](5)--(6);
\draw[Edge](5)--(7);
\draw[Edge](5)--(8);
\draw[Edge](6)--(7);
\draw[Edge](6)--(8);
\draw[Edge](7)--(8);

\end{tikzpicture}
\begin{tikzpicture}[yscale = 0.7,xscale=0.63]
\tikzstyle{Vertex}=[inner sep=1.5pt,thick,shape=circle,draw=black,fill=white];
\tikzstyle{Edge}=[thick];
\node[Vertex] (2) at (0,{(7/6)*sin(210)}) {};
\node[Vertex] (1) at (234:2) {};
\node[Vertex] (0) at (18:2) {};
\node[Vertex] (3) at (90:2){};
\node[Vertex] (4) at (306:2) {};
\node[Vertex] (5) at (162:2) {};
\node[Vertex] (6) at (-30:{7/6}) {};
\node[Vertex] (8) at (90:0.5) {};
\node[Vertex] (7) at (210:{7/6}) {};

\draw[Edge](0)--(3);
\draw[Edge](0)--(4);
\draw[Edge](0)--(6);
\draw[Edge](0)--(7);
\draw[Edge](0)--(8);
\draw[Edge](1)--(4);
\draw[Edge](1)--(5);
\draw[Edge](1)--(6);
\draw[Edge](1)--(7);
\draw[Edge](1)--(8);
\draw[Edge](2)--(6);
\draw[Edge](2)--(7);
\draw[Edge](3)--(5);
\draw[Edge](3)--(6);
\draw[Edge](3)--(7);
\draw[Edge](3)--(8);
\draw[Edge](4)--(6);
\draw[Edge](4)--(7);
\draw[Edge](4)--(8);
\draw[Edge](5)--(6);
\draw[Edge](5)--(7);
\draw[Edge](5)--(8);
\draw[Edge](6)--(8);
\draw[Edge](7)--(8);
\end{tikzpicture} &
\begin{tikzpicture}[yscale = 0.5,xscale=0.45]
\tikzstyle{Vertex}=[inner sep=1.5pt,thick,shape=circle,draw=black,fill=white];
\tikzstyle{Edge}=[thick];
\node[Vertex] (0) at (18:2) {};
\node[Vertex] (1) at (234:2) {};
\node[Vertex] (2) at (0,{2*sin(306)-sqrt(2)}) {};
\node[Vertex] (3) at ({4*cos(306)},{2*sin(306) - sqrt(2)}){};
\node[Vertex] (4) at ({2*cos(234)},{2 + sqrt(2)}) {};
\node[Vertex] (5) at ({2*cos(306},{2+sqrt(2)}) {};
\node[Vertex] (6) at (162:2) {};
\node[Vertex] (8) at (90:2) {};
\node[Vertex] (7) at (306:2) {};

\draw[Edge](0)--(6);
\draw[Edge](0)--(7);
\draw[Edge](0)--(8);
\draw[Edge](1)--(6);
\draw[Edge](1)--(7);
\draw[Edge](1)--(8);
\draw[Edge](2)--(7);
\draw[Edge](3)--(7);
\draw[Edge](4)--(8);
\draw[Edge](5)--(8);
\draw[Edge](6)--(8);
\draw[Edge](7)--(8);
\end{tikzpicture}
\begin{tikzpicture}[yscale = 0.5,xscale=0.45]
\tikzstyle{Vertex}=[inner sep=1.5pt,thick,shape=circle,draw=black,fill=white];
\tikzstyle{Edge}=[thick];
\node[Vertex] (6) at (18:2) {};
\node[Vertex] (7) at (234:2) {};
\node[Vertex] (0) at (0,{2+sqrt(2}) {};
\node[Vertex] (5) at (0,{2 + 2*sqrt(2)}){};
\node[Vertex] (4) at ({2*cos(234)},{2 + sqrt(2)}) {};
\node[Vertex] (3) at ({2*cos(306},{2+sqrt(2)}) {};
\node[Vertex] (1) at (162:2) {};
\node[Vertex] (8) at (90:2) {};
\node[Vertex] (2) at (306:2) {};

\draw[Edge](0)--(5);
\draw[Edge](0)--(8);
\draw[Edge](1)--(6);
\draw[Edge](1)--(7);
\draw[Edge](1)--(8);
\draw[Edge](2)--(6);
\draw[Edge](2)--(7);
\draw[Edge](2)--(8);
\draw[Edge](3)--(8);
\draw[Edge](4)--(8);
\draw[Edge](6)--(8);
\draw[Edge](7)--(8);

\end{tikzpicture} 
}

\def\PICE{
\begin{tikzpicture}[scale = 1]
\tikzstyle{Vertex}=[inner sep=1.5pt,thick,shape=circle,draw=black,fill=white];
\tikzstyle{Edge}=[thick];

\node[Vertex] (1) at (0,0) {\scriptsize$\mathbf1$};
\node[Vertex] (2) at (2,0) {\scriptsize$\mathbf2$};
\node[Vertex] (3) at (2,2) {\scriptsize$\mathbf3$};
\node[Vertex] (4) at (1,0) {\scriptsize$\mathbf4$};
\node[Vertex] (5) at (1,1) {\scriptsize$\mathbf5$};
\node[Vertex] (6) at (0,2) {\scriptsize$\mathbf6$};
\node[Vertex] (7) at (0,1) {\scriptsize$\mathbf7$};
\node[Vertex] (8) at (2,1) {\scriptsize$\mathbf8$};

\draw[Edge] (1)--(4)--(2)--(8)--(5)--(7)--(1) (4)--(6) (7)--(6)--(8)--(3)--(5)--(1) (2)--(5)--(4);

\end{tikzpicture}
&
\begin{tikzpicture}[scale = 1]
\tikzstyle{Vertex}=[inner sep=1.5pt,thick,shape=circle,draw=black,fill=white];
\tikzstyle{Edge}=[thick];

\node[Vertex] (1) at (2,0) {\scriptsize$\mathbf1$};
\node[Vertex] (2) at (0,0) {\scriptsize$\mathbf2$};
\node[Vertex] (3) at (0,2) {\scriptsize$\mathbf3$};
\node[Vertex] (4) at (1,0) {\scriptsize$\mathbf4$};
\node[Vertex] (5) at (2,2) {\scriptsize$\mathbf5$};
\node[Vertex] (6) at (1,1) {\scriptsize$\mathbf6$};
\node[Vertex] (7) at (0,1) {\scriptsize$\mathbf7$};
\node[Vertex] (8) at (2,1) {\scriptsize$\mathbf8$};

\draw[Edge] (2)--(4)--(1)--(8)--(5)--(7)--(3)--(6)--(8)--(4)--(7)--(2)--(6)--(1) (6)--(7);

\end{tikzpicture}

}

\def\PICF{
\begin{tabular}{|c|c|c|c|}\hline
&&&\\[-10pt]
\begin{tikzpicture}[scale = 1]
\tikzstyle{Vertex}=[inner sep=1pt,thick,shape=circle,draw=black,fill=white];
\foreach \x in {1,2,...,9}{ \node[Vertex] (\x) at ({90-(\x-1)*360/11}:1.5) {\tiny$\mathbf{\x}$}; }
\foreach \x in {10,11}{ \node[Vertex,inner sep = 0pt] (\x) at ({90-(\x-1)*360/11}:1.5) {\tiny$\mathbf{\x}$}; }
\draw[thick] (1)--(4) (2)--(5) (3)--(6) (4)--(8) (4)--(9) (5)--(9) (5)--(11) (6)--(9) (6)--(10) (7)--(8) (7)--(9) (7)--(10) (7)--(11) ;
\end{tikzpicture}
&
\begin{tikzpicture}[scale = 1]
\tikzstyle{Vertex}=[inner sep=1pt,thick,shape=circle,draw=black,fill=white];
\foreach \x in {1,2,...,9}{ \node[Vertex] (\x) at ({90-(\x-1)*360/11}:1.5) {\tiny$\mathbf{\x}$}; }
\foreach \x in {10,11}{ \node[Vertex,inner sep = 0pt] (\x) at ({90-(\x-1)*360/11}:1.5) {\tiny$\mathbf{\x}$}; }
\draw[thick] (1)--(8) (1)--(9) (2)--(9) (2)--(11) (3)--(9) (3)--(10) (4)--(5) (4)--(6) (4)--(7) (5)--(6) (5)--(7) (6)--(7) (8)--(9) (8)--(10) (8)--(11) (9)--(10) (9)--(11) (10)--(11)  ;
\end{tikzpicture}
&
\begin{tikzpicture}[scale = 1]
\tikzstyle{Vertex}=[inner sep=1pt,thick,shape=circle,draw=black,fill=white];
\foreach \x in {1,2,...,9}{ \node[Vertex] (\x) at ({90-(\x-1)*360/11}:1.5) {\tiny$\mathbf{\x}$}; }
\foreach \x in {10,11}{ \node[Vertex,inner sep = 0pt] (\x) at ({90-(\x-1)*360/11}:1.5) {\tiny$\mathbf{\x}$}; }
\draw[thick] (1)--(5) (1)--(6) (1)--(7) (2)--(4) (2)--(6) (2)--(7) (3)--(4) (3)--(5) (3)--(7) (4)--(10) (4)--(11) (5)--(8) (5)--(10) (6)--(8) (6)--(11)  ;
\end{tikzpicture}
&
\begin{tikzpicture}[scale = 1]
\tikzstyle{Vertex}=[inner sep=1pt,thick,shape=circle,draw=black,fill=white];
\foreach \x in {1,2,...,9}{ \node[Vertex] (\x) at ({90-(\x-1)*360/11}:1.5) {\tiny$\mathbf{\x}$}; }
\foreach \x in {10,11}{ \node[Vertex,inner sep = 0pt] (\x) at ({90-(\x-1)*360/11}:1.5) {\tiny$\mathbf{\x}$}; }
\draw[thick] (1)--(2) (1)--(3) (1)--(10) (1)--(11) (2)--(3) (2)--(8) (2)--(10) (3)--(8) (3)--(11)  ;
\end{tikzpicture}
\\
$G_1^{(1)}$&$G_1^{(2)}$&$G_1^{(3)}$&$G_1^{(4)}$
\\[5pt]\hline
&&&\\[-10pt]
\begin{tikzpicture}[scale = 1]
\tikzstyle{Vertex}=[inner sep=1pt,thick,shape=circle,draw=black,fill=white];
\foreach \x in {1,2,...,9}{ \node[Vertex] (\x) at ({90-(\x-1)*360/11}:1.5) {\tiny$\mathbf{\x}$}; }
\foreach \x in {10,11}{ \node[Vertex,inner sep = 0pt] (\x) at ({90-(\x-1)*360/11}:1.5) {\tiny$\mathbf{\x}$}; }
\draw[thick] (1)--(4) (2)--(5) (3)--(6) (4)--(8) (4)--(11) (5)--(8) (5)--(10) (6)--(10) (6)--(11) (7)--(8) (7)--(9) (7)--(10) (7)--(11) ;
\end{tikzpicture}
&
\begin{tikzpicture}[scale = 1]
\tikzstyle{Vertex}=[inner sep=1pt,thick,shape=circle,draw=black,fill=white];
\foreach \x in {1,2,...,9}{ \node[Vertex] (\x) at ({90-(\x-1)*360/11}:1.5) {\tiny$\mathbf{\x}$}; }
\foreach \x in {10,11}{ \node[Vertex,inner sep = 0pt] (\x) at ({90-(\x-1)*360/11}:1.5) {\tiny$\mathbf{\x}$}; }
\draw[thick] (1)--(8) (1)--(11) (2)--(8) (2)--(10) (3)--(10) (3)--(11) (4)--(5) (4)--(6) (4)--(7) (5)--(6) (5)--(7) (6)--(7) (8)--(9) (8)--(10) (8)--(11) (9)--(10) (9)--(11) (10)--(11) ;
\end{tikzpicture}
&
\begin{tikzpicture}[scale = 1]
\tikzstyle{Vertex}=[inner sep=1pt,thick,shape=circle,draw=black,fill=white];
\foreach \x in {1,2,...,9}{ \node[Vertex] (\x) at ({90-(\x-1)*360/11}:1.5) {\tiny$\mathbf{\x}$}; }
\foreach \x in {10,11}{ \node[Vertex,inner sep = 0pt] (\x) at ({90-(\x-1)*360/11}:1.5) {\tiny$\mathbf{\x}$}; }
\draw[thick] (1)--(5) (1)--(6) (1)--(7) (2)--(4) (2)--(6) (2)--(7) (3)--(4) (3)--(5) (3)--(7) (4)--(9) (4)--(10) (5)--(9) (5)--(11) (6)--(8) (6)--(9) ;
\end{tikzpicture}
&
\begin{tikzpicture}[scale = 1]
\tikzstyle{Vertex}=[inner sep=1pt,thick,shape=circle,draw=black,fill=white];
\foreach \x in {1,2,...,9}{ \node[Vertex] (\x) at ({90-(\x-1)*360/11}:1.5) {\tiny$\mathbf{\x}$}; }
\foreach \x in {10,11}{ \node[Vertex,inner sep = 0pt] (\x) at ({90-(\x-1)*360/11}:1.5) {\tiny$\mathbf{\x}$}; }
\draw[thick] (1)--(2) (1)--(3) (1)--(9) (1)--(10) (2)--(3) (2)--(9) (2)--(11) (3)--(8) (3)--(9) ;
\end{tikzpicture}
\\
$G_2^{(1)}$&$G_2^{(2)}$&$G_2^{(3)}$&$G_2^{(4)}$\\[5pt]\hline
\end{tabular}
}

\title{Coalescing sets preserving cospectrality of graphs\\ arising from block similarity matrices}
\author{
Sajid Bin Mahamud\footnote{Reed College, Portland, OR 97202, USA, \texttt{sajidbmahamud@reed.edu}} \and
Steve Butler\footnote{Iowa State University, Ames, IA 50011, USA, \texttt{\{butler,nlayman\}@iastate.edu}}\and 
Hannah Graff\footnote{Creighton University, Omaha, NE 68178, USA, \texttt{hannahgraff2@gmail.com}}\and
Nick Layman\footnotemark[2]\and 
Taylor Luck\footnote{Xavier University, Cincinnati, OH 45207, USA, \texttt{taylorluck108@gmail.com}}\and
Jiah Jin\footnote{Cooper Union, New York, NY 10003, USA \texttt{jiah.jin@cooper.edu}}\and
Noah Owen\footnote{University of Kentucky, Lexington, KY 40506, USA, \texttt{naow224@uky.edu}} \and
Angela Yuan\footnote{University of Texas at Austin, Austin, TX 78712, USA \texttt{angelayuan@utexas.edu}}
}
\date{\empty}

\begin{document}

\maketitle

\begin{abstract}
Coalescing involves gluing one or more rooted graphs onto another graph. Under specific conditions, it is possible to start with cospectral graphs that are coalesced in similar ways that will result in new cospectral graphs. We present a sufficient condition for this based on the block structure of similarity matrices, possibly with additional constraints depending on which type of matrix is being considered. The matrices considered in this paper include the adjacency, Laplacian, signless Laplacian, distance, and generalized distance matrix.
\end{abstract}

\section{Introduction}\label{sec:intro}
A famous problem in spectral graph theory is whether it is possible to ``hear the shape of a graph,'' meaning whether a graph can be uniquely identified by its set of eigenvalues associated with a matrix. There are many cases of pairs of non-isomorphic graphs that have the same set of eigenvalues. Such graphs are called cospectral, and so the answer in general is ``no'', though whether you can hear the shape of a graph for almost all graphs is still an open problem (see \cite{MR2070541,MR2022290,MR2499010}).

A common method for constructing large cospectral graphs is to start with small cospectral graphs and augment them in some way. One popular method is known as coalescing, which can be thought of as gluing graphs together at specified vertices (see Section~\ref{sec:coalescing}). This was used by Schwenk \cite{Schwenk} to establish that almost all trees have a cospectral mate for the adjacency matrix, which was later extended by McKay \cite{McKay} to the Laplacian, signless Laplacian, and distance matrices. Coalescing has also been used for other constructions, e.g.\ Heysse \cite{Heysse} used coalescing to form cospectral pairs for the distance matrix with differing numbers of edges.

While coalescing is well-understood for certain matrix families (e.g.\ adjacency, Laplacian, and signless Laplacian; see \cite{REU22,guo}) which can use sparseness to give combinatorial methods to compute characteristic polynomials, it is poorly understood for other matrix families which are dense (e.g.\ distance). While coalescing has been previously used on the distance matrix, in those cases the arguments relied heavily on special graph structures for the characteristic polynomial (see McKay \cite{McKay}) or for eigenvector arguments (see Heysse \cite{Heysse}). 

The goal of this paper is to establish sufficient conditions for constructing cospectral graphs by use of coalescing in terms of the structure of the similarity matrices of the base graphs used. We will establish results for the adjacency, Laplacian, signless Laplacian, distance, and generalized distance matrices. In particular, this paper marks significant progress in understanding coalescing for distance matrices and the construction of cospectral graphs for distance matrices (see \cite{distancesurvey,Aouchiche,Survey2}). 

We also answer ``no'' to the following conjecture  related to the distance matrix which was the original motivation for the research resulting in this paper.
\begin{conjecture}[{Butler et al.\ \cite{REU22}}]
Let $G_1$ and $G_2$ be two graphs with $B_1\subseteq V(G_1)$ and $B_2\subseteq V(G_2)$ such that coalescing the same (connected) rooted graph onto all the vertices of $B_1$ and $B_2$ \emph{always} results in cospectral pairs for the distance matrix. Then, coalescing the same (connected) rooted graph onto all the vertices of $V\setminus B_1$ and $V\setminus B_2$ will also always result in cospectral pairs for the distance matrix.
\end{conjecture}
See Figure~\ref{fig:nine} for graphs where this is not true (in this case $B_1=B_2=\emptyset$).

\bigskip

In the remainder of the introduction we will formally define the matrices associated with graphs and give some relevant definitions that will be useful. In Section~\ref{sec:coalescing} we will formally define coalescing and discuss how to use block structure of the similarity matrix to label graphs and their coalescings. In Section~\ref{sec:main} we will state and establish our main results, which is followed in Section~\ref{sec:examples} of various examples and applications of our results.

\subsection{Basic definitions}
In this paper we will assume that all graphs are simple (no loops or weighted edges) and undirected.

\begin{definition}
For a vertex $v$ in a graph $G$, the \emph{degree} of $v$, denoted $\deg_G(v)$, is the number of neighbor vertices of $v$.

For a pair of vertices $u,v$ in a graph $G$, the \emph{distance} between $u$ and $v$, denoted $\dist_G(u,v)$, is the length of the shortest path between $u$ and $v$. If $u=v$, then $\dist_G(u,v)=0$. If $u$ and $v$ are in separate components, then $\dist_G(u,v)=\infty$.
\end{definition}

We will consider the following three different matrices associated with graphs.

\begin{definition}
Let $q$ be fixed. The \emph{$q$-Laplacian matrix} of a graph $G$, denoted $L^{(q)}_G$ or $L^{(q)}$ when $G$ is clear from context, is defined entry-wise by
\[
L^{(q)}_{u,v}=\begin{cases}
q\deg_G(u)&\text{if $u=v$,}\\
1&\text{if $u$ and $v$ are adjacent in $G$,}\\
$0$&\text{otherwise.}
\end{cases}
\]
\end{definition}

We note that $L^{(0)}$ is the adjacency matrix, $L^{(1)}$ is the signless Laplacian, and $L^{(-1)}$ is the negative of the Laplacian. We use this convention as the proof of the main result is identical for these three matrices.

\begin{definition}
The \emph{distance matrix} of a connected graph $G$, denoted $\mathcal{D}_G$ or $\mathcal{D}$ when $G$ is clear from context, is defined entry-wise by $\mathcal{D}_{u,v}=\dist_G(u,v)$.

Given a function $f:\{0,1,\ldots\}\to \mathbb{R}$, the \emph{generalized distance matrix} of a connected graph $G$, denoted $\mathcal{D}^{f}_G$ or $\mathcal{D}^f$ when $G$ is clear from context, is defined entry-wise by $\mathcal{D}^{f}_{u,v}=f(\dist_G(u,v))$.
\end{definition}

Examples of generalized distance matrices that have been studied before include the squared distance matrix with entries $\big(\dist(u,v)\big)^2$ (see \cite{Bapat}), and the exponential distance matrix with entries $q^{\dist(u,v)}$ (see \cite{MR2242465,REU19}).

We will denote a block diagonal matrix by 
\[
B_1\oplus B_2\oplus\cdots\oplus B_\ell=\begin{pmatrix}
B_1&O&\cdots&O\\
O&B_2&\cdots&O\\
\vdots&\vdots&\ddots&\vdots\\
O&O&\cdots&B_\ell
\end{pmatrix}.
\]
Moreover, for a matrix $M$, we will let $M[s,t]$ denote the submatrix consisting of the rows in $s$ and the columns in $t$, where $s$, $t$ are subsets of indices of rows/columns. We will use $I$ to denote the identity matrix, $J$ to denote the all $1$'s matrix, and $O$ the all $0$'s matrix. For notational convenience we will suppress the notation about the sizes of $I$, $J$, and $O$ as that can be determined from context.

\section{Coalescing and graph labeling}\label{sec:coalescing}
As introduced earlier, we can think of coalescing graphs as gluing two graphs together on a common vertex. It is also possible to coalesce several graphs simultaneously. More formally, we present the following definition:

\begin{definition}
Given graphs $G$ with vertex $v$ and a rooted graph $H$, the \emph{coalescing of $H$ onto $v$}, denoted $G\coal_vH$, is the graph formed by taking the disjoint union of the graphs $G$ and $H$ and then identifying the root of $H$ and $v$ as the same vertex. 

Given a subset $V_i\subseteq V(G)$ and a rooted graph $H$, the \emph{coalescing of $H$ onto $V_i$}, denoted by $G\coal_{V_i}H$, is the graph constructed by taking the disjoint union of the graph $G$ and $|V_i|$ copies of $H$, and then for each $v\in V_i$ taking a unique copy of $H$ and identify the root of that $H$ and $v$ as the same vertex.

Given $V_1,\ldots,V_\ell\subseteq V(G)$ and rooted graphs $H_1,\ldots,H_\ell$, we will let $\displaystyle G\coal_{i=1}^\ell{}_{V_i}H_i$ denote the graph formed by the coalescing of $H_i$ onto $V_i$ for $i=1,\ldots,\ell$.
\end{definition}

When we coalesce the graph $K_1$ (the graph of a single vertex) onto a vertex $v$, the result does not change the graph structure at $v$. In particular, we can always treat every coalescing of multiple graphs by taking the vertices of $V(G)$ into a disjoint partition of sets $V_1,V_2,\ldots,V_\ell$ and then forming $\displaystyle G\coal_{i=1}^\ell{}_{V_i}H_i$ where $H_i$ might possibly be $K_1$.

\subsection{Labeling vertices in a coalescing of graphs}
A useful ingredient in establishing our main result will be using our vertex partition $V_1,\ldots,V_\ell$ to give a labeling for our graph $G$ and the corresponding graph $\displaystyle G\coal_{i=1}^\ell{}_{V_i}H_i$.

For the graph $G$, we label the vertices as $i{:}1{:}k$ which indicates the $k$-th vertex of $V_i$ with $1\le k\le |V_i|$. An example of this labeling is seen in Figure~\ref{fig:coalescing}(a) where the blocks $V_1,V_2,V_3$ are indicated with rectangles. Now label the vertices of $H_i$ with $1,\ldots,|V(H_i)|$ where $1$ is the root (the vertex we coalesce on). An example of this labeling is seen in Figure~\ref{fig:coalescing}(b). Then the labeling for $\displaystyle G\coal_{i=1}^\ell{}_{V_i}H_i$ will be $i{:}j{:}k$ which indicates vertex $j$ in the copy of $H_i$ that was coalesced onto the vertex $i{:}1{:}k$. We will use the notation $i{:}j$ to denote the set of vertices $\{i{:}j{:}k \,| \, k \in V_i\}$. 
Finally, we group the vertices of $\displaystyle G\coal_{i=1}^\ell{}_{V_i}H_i$ into blocks of the form $i{:}j$. An example of this grouping into blocks is shown in Figure~\ref{fig:coalescing}(c). We note that with this convention that the blocks of the form $i{:}1$ correspond with the blocks of the graph $G$ corresponding with our vertex partition.

\begin{figure}[htb]\centering
\PICcoalesce
\caption{An example of $\displaystyle G\coal_{i=1}^3{}_{V_i}H_i$ with labeling and block structures marked.}
\label{fig:coalescing}
\end{figure}

It will be informative to compare the distance matrix for the graphs $G$ and $\displaystyle G\coal_{i=1}^3{}_{V_i}H_i$ from Figure~\ref{fig:coalescing}. These are as follows, where we have labeled the blocks of the corresponding graphs using the convention outlined above. (We associate the vertices with the rows/columns by arranging them in lexicographical order.)

\medskip

\begin{centering}
\begin{tabular}{c@{\quad}c}
\begin{tabular}{c}
\footnotesize$\displaystyle
\begin{array}{c}
1{:}1\\
\\
\\ \hline
2{:}1 \\ \hline
3{:}1 \\
\\ \end{array}
\left(
\begin{array}{ccc|c|cc}
0&1&2&3&4&4\\
1&0&1&2&3&3\\
2&1&0&1&2&2\\ \hline
3&2&1&0&1&1\\ \hline
4&3&2&1&0&1\\
4&3&2&1&1&0\\ \end{array}
\right)$
\end{tabular}&
\begin{tabular}{c}
\footnotesize$\displaystyle
\begin{array}{c}
1{:}1\\
\\
\\ \hline
1{:}2\\
\\
\\ \hline
1{:}3\\
\\
\\ \hline
2{:}1\\ \hline
3{:}1\\
\\ \hline
3{:}2\\
\\ \hline
3{:}3\\
\\ \end{array}
\left(\begin{array}{ccc|ccc|ccc|c|cc|cc|cc}
0 & 1 & 2 & 1 & 2 & 3 & 2 & 3 & 4 & 3 & 4 & 4 & 5 & 5 & 5 & 5 \\
1 & 0 & 1 & 2 & 1 & 2 & 3 & 2 & 3 & 2 & 3 & 3 & 4 & 4 & 4 & 4 \\
2 & 1 & 0 & 3 & 2 & 1 & 4 & 3 & 2 & 1 & 2 & 2 & 3 & 3 & 3 & 3 \\ \hline
1 & 2 & 3 & 0 & 3 & 4 & 1 & 4 & 5 & 4 & 5 & 5 & 6 & 6 & 6 & 6 \\
2 & 1 & 2 & 3 & 0 & 3 & 4 & 1 & 4 & 3 & 4 & 4 & 5 & 5 & 5 & 5 \\
3 & 2 & 1 & 4 & 3 & 0 & 5 & 4 & 1 & 2 & 3 & 3 & 4 & 4 & 4 & 4 \\ \hline
2 & 3 & 4 & 1 & 4 & 5 & 0 & 5 & 6 & 5 & 6 & 6 & 7 & 7 & 7 & 7 \\
3 & 2 & 3 & 4 & 1 & 4 & 5 & 0 & 5 & 4 & 5 & 5 & 6 & 6 & 6 & 6 \\
4 & 3 & 2 & 5 & 4 & 1 & 6 & 5 & 0 & 3 & 4 & 4 & 5 & 5 & 5 & 5 \\ \hline
3 & 2 & 1 & 4 & 3 & 2 & 5 & 4 & 3 & 0 & 1 & 1 & 2 & 2 & 2 & 2 \\ \hline
4 & 3 & 2 & 5 & 4 & 3 & 6 & 5 & 4 & 1 & 0 & 1 & 1 & 2 & 1 & 2 \\
4 & 3 & 2 & 5 & 4 & 3 & 6 & 5 & 4 & 1 & 1 & 0 & 2 & 1 & 2 & 1 \\ \hline
5 & 4 & 3 & 6 & 5 & 4 & 7 & 6 & 5 & 2 & 1 & 2 & 0 & 3 & 1 & 3 \\
5 & 4 & 3 & 6 & 5 & 4 & 7 & 6 & 5 & 2 & 2 & 1 & 3 & 0 & 3 & 1 \\ \hline
5 & 4 & 3 & 6 & 5 & 4 & 7 & 6 & 5 & 2 & 1 & 2 & 1 & 3 & 0 & 3 \\
5 & 4 & 3 & 6 & 5 & 4 & 7 & 6 & 5 & 2 & 2 & 1 & 3 & 1 & 3 & 0 \\
\end{array}\right)$
\end{tabular}
\\ \\
Distance matrix for $G$&
Distance matrix for $\displaystyle G\coal_{i=1}^3{}_{V_i}H_i$
\end{tabular}
\end{centering}

\medskip

The key observation is to note that the blocks of the two matrices are similar up to shifting in predictable ways. More precisely we have the following.

\begin{lemma}\label{lem:shift}
Let $M$ be the distance matrix for the graph $G$ and let $N$ be the distance matrix for the graph $\displaystyle G\coal_{i=1}^\ell{}_{V_i}H_i$, then
\[
N[i_1{:}j_1,i_2{:}j_2]=
\begin{cases}
M[i_1{:}1,i_2{:}1]+\alpha J&\text{if $i_1 \ne i_2$},\\
M[i_1{:}1,i_2{:}1]+\alpha J+(\beta-\alpha) I&\text{if $i_1 = i_2$},
\end{cases}
\]
where $\alpha=\dist_{H_{i_1}}(1,j_1)+\dist_{H_{i_2}}(1,j_2)$ and $\beta = \dist_{H_{i_1}}(j_1,j_2)$. (Recall that $1$ corresponds with the root of the corresponding $H_{i_k}$ on which we coalesce.)
\end{lemma}
\begin{proof}
When $i_1\ne i_2$, any path connecting $i_1{:}j_1{:}k_1$ with $i_2{:}j_2{:}k_2$ must travel through the two cut vertices $i_1{:}1{:}k_1$ with $i_2{:}1{:}k_2$ (locations of where coalescing occurred) in the base graph. In particular, the shortest path is naturally split into three parts with the distance as follows:
\[
\dist_{H_{i_1}}(1,j_1)+\dist_G(i_1{:}1{:}k_1,i_2{:}1{:}k_2)+\dist_{H_{i_2}}(1,j_2).
\]
The first and the last terms are independent of the choices of $k_1$ and $k_2$ so can be pulled out as an additive constant ($\alpha$) for each entry, the middle term then comes from $M[i_1{:}1,i_2{:}1]$.

When $i_1=i_2=i$, the result is similar with the exception of $k_1=k_2=k$ (the diagonal terms). In this case case the shortest path does not go through $G$ but stays inside of $H_i$ (that is, we are finding distance internally to the graph $H_i$ that was glued onto ${i}{:}1{:}k$). So we can add $\alpha$ to all entries (the $+\alpha J$) and then on the diagonal subtract out $\alpha$ and add $\beta=\dist_{H_i}(j_1,j_2)$.
\end{proof}

\section{Main results}\label{sec:main}
Our main results involve a sufficient condition for constructing cospectral graphs from coalescing based off of similarity matrices, and in particular block similarity matrices. We start by showing how to use the block structure of the similarity matrix to partition the vertices.

\begin{definition}
Let $S=B_1\oplus \cdots\oplus B_\ell$ be a similarity matrix for $M_{G_1}$ and $M_{G_2}$, matrices associated with the graphs $G_1$ and $G_2$. Then a \emph{similarity vertex partitioning} $V_1, \ldots,V_\ell$ is a partitioning of the vertices of $G_1$ and $G_2$ by associating the vertices in the rows corresponding with block $B_i$ into vertex set $V_i$.
\end{definition}

We now state our main results; each one of them dealing with coalescing and similarity vertex partitioning for different families of graphs.

\begin{theorem}\label{thm:Lq}
If $G_1$ and $G_2$ are cospectral for the $L^{(q)}$ matrix with $S=B_1\oplus \cdots\oplus B_\ell$ as a similarity matrix, then for the similarity vertex partitioning $V_1,\ldots,V_\ell$ the graphs $\displaystyle G_1\coal_{i=1}^\ell{}_{V_i}H_i$ and $\displaystyle G_2\coal_{i=1}^\ell{}_{V_i}H_i$ are cospectral for the $L^{(q)}$ matrix for any choice of $H_1,\ldots,H_\ell$.
\end{theorem}

\begin{theorem}\label{thm:D}
If $G_1$ and $G_2$ are cospectral for the $\mathcal{D}$ matrix with $S=B_1\oplus \cdots\oplus B_\ell$ as a similarity matrix satisfying $SJ=JS$, then for the similarity vertex partitioning $V_1,\ldots,V_\ell$ the graphs $\displaystyle G_1\coal_{i=1}^\ell{}_{V_i}H_i$ and $\displaystyle G_2\coal_{i=1}^\ell{}_{V_i}H_i$ are cospectral for the $\mathcal{D}$ (distance) matrix for any choice of $H_1,\ldots,H_\ell$.
\end{theorem}

The difference between Theorem~\ref{thm:Lq} and Theorem~\ref{thm:D} is that in the latter case we need the assumption $SJ=JS$; we will have more discussion related to this assumption in Section~\ref{sec:examples}.

\begin{theorem}\label{thm:genD}
Given a graph $G$ let $G^{(t)}$ be the graph with $V(G^{(t)})=V(G)$ and $E(G^{(t)})=\big\{\{u,v\}\mid\dist_G(u,v)=t\big\}$. If $S=B_1\oplus \cdots\oplus B_\ell$ is \emph{simultaneously} a similarity matrix for $G_1^{(t)}$ and $G_2^{(t)}$ for the matrix $L^{(0)}$ (adjacency matrix) for $t=0,1,2,\ldots$, then for the similarity vertex partitioning $V_1,\ldots,V_\ell$ the graphs $\displaystyle G_1\coal_{i=1}^\ell{}_{V_i}H_i$ and $\displaystyle G_2\coal_{i=1}^\ell{}_{V_i}H_i$ are cospectral for the $\mathcal{D}^f$ (generalized distance) matrix for any choice of $H_1,\ldots,H_\ell$ and any choice of function $f$.
\end{theorem}

\subsection{Proof of Theorems~\ref{thm:Lq}, \ref{thm:D}, and \ref{thm:genD}}
For convenience we will assume that all graphs are labeled using the conventions given in Section~\ref{sec:coalescing}.

Let $M$ be the matrix that we are considering. We have that $S$ is a similarity matrix so we can assume that $SM_{G_1}S^{-1}=M_{G_2}$, or more conveniently $SM_{G_1}=M_{G_2}S$. We then have that submatrices are equal so $(SM_{G_1})[i_1{:}1,i_2{:}1]=(M_{G_2}S)[i_1{:}1,i_2{:}1]$. Finally, using that $S$ is block diagonal this becomes
\begin{equation}\label{eq:subblock}
B_{i_1}M_{G_1}[i_1{:}1,i_2{:}1]=M_{G_2}[i_1{:}1,i_2{:}1]B_{i_2}.
\end{equation}

Let $\widehat{G_1}=\displaystyle G_1\coal_{i=1}^\ell{}_{V_i}H_i$ and $\widehat{G_2}=\displaystyle G_2\coal_{i=1}^\ell{}_{V_i}H_i$. To show that $M_{\widehat{G_1}}$ is cospectral with $M_{\widehat{G_2}}$ we will produce a similarity matrix $\widehat{S}$ so that $\widehat{S}M_{\widehat{G_1}}=M_{\widehat{G_2}}\widehat{S}$. Our desired similarity matrix is
\[
\widehat{S}=\underbrace{B_1\oplus\cdots\oplus B_1}_{\text{$|V(H_1)|$ times}}\oplus\underbrace{B_2\oplus\cdots\oplus B_2}_{\text{$|V(H_2)|$ times}}\oplus\cdots\oplus\underbrace{B_\ell\oplus\cdots\oplus B_\ell}_{\text{$|V(H_\ell)|$ times}}.
\]
We note in passing that since $S$ is a similarity matrix, it must be invertible so each individual block is square and invertible. Now since $\widehat{S}$ is a block diagonal matrix and each block on the diagonal is square and invertible it itself is also invertible.

To establish $\widehat{S}M_{\widehat{G_1}}=M_{\widehat{G_2}}\widehat{S}$ it suffices to establish it for a collection of submatrices which cover the matrix. In particular, it suffices to show that $(\widehat{S}M_{\widehat{G_1}})[i_1{:}j_1,i_2{:}j_2]=(M_{\widehat{G_2}}\widehat{S})[i_1{:}j_1,i_2{:}j_2]$ for all possible $i_1,j_1,i_2,j_2$. As before, using that $\widehat{S}$ is block diagonal this reduces to verifying
\begin{equation}\label{eq:BIGsubblock}
B_{i_1}M_{\widehat{G_1}}[i_1{:}j_1,i_2{:}j_2]=M_{\widehat{G_2}}[i_1{:}j_1,i_2{:}j_2]B_{i_2}
\end{equation}
to establish the result.

\subsubsection*{Verification of \eqref{eq:BIGsubblock} for Theorem~\ref{thm:Lq}}
For $M=L^{(q)}$ we have the following for $G_1$ (and similar result with $G_2$):
\[
M_{\widehat{G_1}}[i_1{:}j_1,i_2{:}j_2]=
\begin{cases}
M_{G_1}[i_1{:}1,i_1{:}1]+q\deg_{H_{i_1}}(1)I&\text{if $i_1=i_2$ and $j_1=j_2=1$},\\
M_{G_1}[i_1{:}1,i_2{:}1]&\text{if $i_1\ne i_2$ and $j_1=j_2=1$},\\
q\deg_{H_{i_1}}(j_1)I&\text{if $i_1=i_2$ and $j_1=j_2>1$},\\
I&\text{if $i_1=i_2$ and $\{j_1,j_2\}\in E(H_1)$},\\
O&\text{else}.
\end{cases}
\]
In this case we have that \eqref{eq:BIGsubblock} follows in each case by appropriate combination of \eqref{eq:subblock} (for when $j_1=j_2=1$), $B_{i_1}I=IB_{i_1}$, and $B_{i_1}O=O=OB_{i_2}$. We illustrate with the case $i_1=i_2$ and $j_1=j_2=1$, the other cases are handled similarly.
\begin{align*}
B_{i_1}M_{\widehat{G_1}}[i_1{:}j_1,i_2{:}j_2]
&=B_{i_1}\big(M_{G_1}[i_1{:}1,i_1{:}1]+q\deg_{H_{i_1}}(1)I\big)\\
&=B_{i_1}M_{G_1}[i_1{:}1,i_1{:}1]+q\deg_{H_{i_1}}(1)B_{i_1}I\\
&=M_{G_2}[i_1{:}1,i_1{:}1]B_{i_2}+q\deg_{H_{i_1}}(1)IB_{i_2}\\
&=\big(M_{G_2}[i_1{:}1,i_1{:}1]+q\deg_{H_{i_1}}(1)I\big)B_{i_2}\\
&=M_{\widehat{G_2}}[i_1{:}j_1,i_2{:}j_2]B_{i_2}
\end{align*}

\subsubsection*{Verification of \eqref{eq:BIGsubblock} for Theorem~\ref{thm:D}}
From our added assumption that $SJ=JS$ we have that $B_{i_1}J=JB_{i_2}$ which follows using the same argument as \eqref{eq:subblock}. 

For $M=\mathcal{D}$ and Lemma~\ref{lem:shift} we have the following for $G_1$ (and similar result with $G_2$):
\[
M_{\widehat{G_1}}[i_1{:}j_1,i_2{:}j_2]=
\begin{cases}
M_{G_1}[i_1{:}1,i_2{:}1]+\alpha J&\text{if $i_1 \ne i_2$},\\
M_{G_1}[i_1{:}1,i_2{:}1]+\alpha J+(\beta-\alpha) I&\text{if $i_1 = i_2$},
\end{cases}
\]
where $\alpha=\dist_{H_{i_1}}(1,j_1)+\dist_{H_{i_2}}(1,j_2)$ and $\beta = \dist_{H_{i_1}}(j_1,j_2)$. Note that $\beta$ and $\alpha$ are independent of $G_1$ and $G_2$.

In this case we have that \eqref{eq:BIGsubblock} follows in each case by appropriate combination of \eqref{eq:subblock}, $B_{i_1}I=IB_{i_1}$, and $B_{i_1}J=JB_{i_2}$. We illustrate with the case $i_1=i_2$, the other case is handled similarly.
\begin{align*}
B_{i_1}M_{\widehat{G_1}}[i_1{:}j_1,i_2{:}j_2]
&=B_{i_1}\big(M_{G_1}[i_1{:}1,i_2{:}1]+\alpha J+(\beta-\alpha) I\big)\\
&=B_{i_1}M_{G_1}[i_1{:}1,i_2{:}1]+\alpha B_{i_1}J+(\beta-\alpha) B_{i_1}I\\
&=M_{G_2}[i_1{:}1,i_2{:}1]B_{i_2}+\alpha JB_{i_2}+(\beta-\alpha) IB_{i_2}\\
&=\big(M_{G_2}[i_1{:}1,i_2{:}1]+\alpha J+(\beta-\alpha) I\big)B_{i_2}\\
&=M_{\widehat{G_2}}[i_1{:}j_1,i_2{:}j_2]B_{i_2}
\end{align*}

\subsubsection*{Verification of \eqref{eq:BIGsubblock} for Theorem~\ref{thm:genD}}
For $M=\mathcal{D}^f$ we will let $A_{G_p}^{(t)}$ denote the adjacency matrix ($L^{(0)}$) for $G_p^{(t)}$. Since $S$ is a similarity matrix for the adjacency matrices for $G_1^{(t)}$ and $G_2^{(t)}$ it follows that $SA_{G_1}^{(t)}=A_{G_2}^{(t)}S$ for $t=0,1,2,\ldots$. By the same argument as \eqref{eq:subblock} we have for all $t=0,1,2,\ldots$
\begin{equation}\label{eq:genD}
B_{i_1}A_{G_1}^{(t)}[i_1{:}1,i_2{:}1]=A_{G_2}^{(t)}[i_1{:}1,i_2{:}1]B_{i_2}.
\end{equation}

For $G_1$ (and similarly $G_2$) we have:
\[
M_{G_1}[i_1{:}1,i_2{:}1]=\sum_{t\ge0}f(t)A_{G_1}^{(t)}[i_1{:}1,i_2{:}1].
\]
Because entries are based off of distance, we can apply Lemma~\ref{lem:shift} and note that internally in our blocks that distances all shift in uniform ways. In particular, we have the following for $G_1$ (and similar result with $G_2$):
\[
M_{\widehat{G_1}}[i_1{:}j_1,i_2{:}j_2]=
\begin{cases}
\displaystyle\sum_{t\ge 0}f(t+\alpha)A_{G_1}^{(t)}[i_1{:}1,i_2{:}1]&\text{if $i_1 \ne i_2$},\\[17pt]
\displaystyle\sum_{t\ge 1}f(t+\alpha)A_{G_1}^{(t)}[i_1{:}1,i_2{:}1]+f(\beta)I&\text{if $i_1 = i_2$},
\end{cases}
\]
where $\alpha=\dist_{H_{i_1}}(1,j_1)+\dist_{H_{i_2}}(1,j_2)$ and $\beta = \dist_{H_{i_1}}(j_1,j_2)$. Note that $\beta$ and $\alpha$ are independent of $G_1$ and $G_2$.

In this case we have that \eqref{eq:BIGsubblock} follows in each case by appropriate combination of \eqref{eq:genD} and $B_{i_1}I=IB_{i_1}$. We illustrate with the case $i_1=i_2$, the other case is handled similarly.
\begin{align*}
B_{i_1}M_{\widehat{G_1}}[i_1{:}j_1,i_2{:}j_2]
&=B_{i_1}\bigg(\sum_{t\ge 1}f(t+\alpha)A_{G_1}^{(t)}[i_1{:}1,i_2{:}1]+f(\beta)I\bigg)\\
&=\sum_{t\ge 1}f(t+\alpha)B_{i_1}A_{G_1}^{(t)}[i_1{:}1,i_2{:}1]+f(\beta)B_{i_1}I\\
&=\sum_{t\ge 1}f(t+\alpha)A_{G_2}^{(t)}[i_1{:}1,i_2{:}1]B_{i_2}+f(\beta)IB_{i_2}\\
&=\bigg(\sum_{t\ge 1}f(t+\alpha)A_{G_2}^{(t)}[i_1{:}1,i_2{:}1]+f(\beta)I\bigg)B_{i_2}\\
&=M_{\widehat{G_2}}[i_1{:}j_1,i_2{:}j_2]B_{i_2}
\end{align*}

\section{Comments and applications}\label{sec:examples}
In order to apply Theorems~\ref{thm:Lq}, \ref{thm:D}, and \ref{thm:genD}, one must first find a block diagonal similarity matrix. It is not always obvious, a priori, when such block diagonal similarity matrices exist. However, the results of this paper suggest an approach to potentially finding them. Namely, by experimentation, e.g.\ using several small graphs such as stars, look for subsets of vertices where coalescing on the same set produces cospectral graphs. This leads to a potential vertex partitioning which we can use to add additional constraints when looking for block diagonal similarity matrices. This was the technique that the authors used in their exploration of small cases.

We now give some comments, a few examples, and applications of the results. For many of the graphs which are shown, we will indicate the corresponding \verb!graph6! code for reference for anyone wanting to input the graphs into a computer.

\subsubsection*{A sufficient but not necessary condition}
The results of this paper give a sufficient condition for when we can coalesce, however this is not a necessary condition. In Butler et al.\ \cite{REU22} necessary \emph{and} sufficient conditions were derived in the case $L^{(q)}$. Here we state a simple version in terms of our notation for the adjacency matrix with vertex partitioning $V(G)=V_1\cup V_2$.

\begin{theorem}[Butler et al.\ \cite{REU22}]\label{thm:necessary}
Given a graph $H$ and $U\subseteq V(H)$, let $p_{H,U}(x)$ denote the characteristic polynomial of the adjacency matrix for the induced subgraph in $H$ on the vertices $U$. Given graphs $G_1$ and $G_2$ with the same labeling of vertices and partition $V=V_1\cup V_2$, then the graphs $\displaystyle G_1\coal_{i=1}^2{}_{V_i}H_i$ and $\displaystyle G_2\coal_{i=1}^2{}_{V_i}H_i$ are cospectral for the adjacency matrix for any choice of $H_1$ and $H_2$ if and only if for all $k,\ell$ we have
\[
\sum_{\substack{S\subseteq V_1,|S|=k\\
T\subseteq V_2, |T|=\ell}}p_{G_1,S\cup T}(x)=\sum_{\substack{S\subseteq V_1,|S|=k\\
T\subseteq V_2, |T|=\ell}}p_{G_2,S\cup T}(x).
\]
\end{theorem}

In Figure~\ref{fig:necessary} we give two graphs on $7$ vertices for which Theorem~\ref{thm:necessary} can be used to establish $\displaystyle G_1\coal_{i=1}^2{}_{V_i}H_i$ and $\displaystyle G_2\coal_{i=1}^2{}_{V_i}H_i$ are cospectral for any choice of $H_1$ and $H_2$ with $V_1=\{1,2,3\}$ and $V_2=\{4,5,6,7\}$. However, Theorem~\ref{thm:Lq} does not hold as there is no block diagonal similarity matrix of the correct form (this can be established by setting up a series of equations such a similarity matrix would need to satisfy and showing the system is inconsistent). So there are limitations to the coalescing results presented in this paper; nevertheless, it does establish significantly many cases, and it is often easier to find a block similarity matrix than to check polynomial conditions.

\begin{figure}[htb]
\centering
\begin{tabular}{cc}
\PICAA & \PICAB \\
$G_1 =$ \verb!F@AMw!&$G_2 =$ \verb!F@AZg!
\end{tabular}
\caption{$\displaystyle G_1\coal_{i=1}^2{}_{V_i}H_i$ and $\displaystyle G_2\coal_{i=1}^2{}_{V_i}H_i$ are cospectral for the adjacency matrix any choice of $H_1$ and $H_2$ with $V_1=\{1,2,3\}$ and $V_2=\{4,5,6,7\}$.}
\label{fig:necessary}
\end{figure}

For Theorem~\ref{thm:D}, the authors found several examples for the distance matrix of graphs where experimentally $\displaystyle G_1\coal_{i=1}^\ell{}_{V_i}H_i$ and $\displaystyle G_2\coal_{i=1}^\ell{}_{V_i}H_i$ are distance-cospectral for arbitrary choice of $H_1,\ldots,H_\ell$, but no block similarity matrix satisfying the requirements of the theorem exists. One such example is shown in Figure~\ref{fig:Dnecessary}.

\begin{figure}[htb]
\centering
\begin{tabular}{cc}
\PICBA & \PICBB\\
$G_1 =$ \verb!GNKutO!&$G_2 =$ \verb!GB}XV_!
\end{tabular}
\caption{Experimentally, $\displaystyle G_1\coal_{i=1}^4{}_{V_i}H_i$ and $\displaystyle G_2\coal_{i=1}^4{}_{V_i}H_i$ are cospectral for the distance matrix for many choices of $H_1,$ $H_2, H_3, H_4$ with $V_1=\{1\}$, $V_2=\{2\}$, $V_3=\{3,4,5\}$ and $V_4=\{6,7,8\}$.}
\label{fig:Dnecessary}
\end{figure}

\begin{question}
For the graphs shown in Figure~\ref{fig:Dnecessary} are $\displaystyle G_1\coal_{i=1}^4{}_{V_i}H_i$ and $\displaystyle G_2\coal_{i=1}^4{}_{V_i}H_i$ cospectral for arbitrary choice of $H_1,H_2,H_3$,and $H_4$?
\end{question}

We note that for the graphs in Figure~\ref{fig:Dnecessary}, the following is a block diagonal similarity matrix for the distance matrix satisfying the conditions of Theorem~\ref{thm:D}, showing that $\displaystyle G_1\coal_{i=1}^3{}_{V_i}H_i$ and $\displaystyle G_2\coal_{i=1}^3{}_{V_i}H_i$ are cospectral for the distance matrix for arbitrary choice of $H_1,$ $H_2,$ and  $H_3$ with $V_1=\{1\}$, $V_2=\{2\}$, $V_3=\{3,4,5,6,7,8\}$. 
\[
S= \begin{pmatrix}1\end{pmatrix}\oplus\begin{pmatrix}1\end{pmatrix}\oplus\frac12\left(\begin{array}{rrrrrr}
0&1&1&-1&1&0\\
1&0&1&0&-1&1\\
1&1&0&1&0&-1\\
1&0&-1&0&1&1\\
-1&1&0&1&0&1\\
0&-1&1&1&1&0
\end{array}\right)
\]
This answers the previous question in the affirmative in the case where $H_3 = H_4$.

\begin{question}
Given $G_1$ and $G_2$, is there a necessary and sufficient condition to establish $\displaystyle G_1\coal_{i=1}^\ell{}_{V_i}H_i$ and $\displaystyle G_2\coal_{i=1}^\ell{}_{V_i}H_i$ are cospectral for the distance matrix for any choice of $H_1,\ldots,H_\ell$?
\end{question}

\subsubsection*{Alternative proofs of earlier results for coalescing on the distance matrix}
Coalescing for the distance matrix has previously appeared in the literature in two cases. We now show how Theorem~\ref{thm:D} can be used to establish those results.

McKay \cite{McKay} used coalescing on trees to show that almost all trees have a cospectral mate. Building off of Schwenk \cite{Schwenk} the key was finding a pair of cospectral trees for $\mathcal{D}$ which had a vertex for which coalescing an arbitrary tree at that vertex resulted in cospectral trees for $\mathcal{D}$ which with high probability were non-isomorphic. McKay's proof is based on characteristic polynomial arguments which relied heavily on the fact that a tree was involved.

\begin{theorem}[McKay \cite{McKay}]
For the graphs shown in Figure~\ref{fig:McKay}, $G_1\coal_1H$ is cospectral with $G_2\coal_1H$ with respect to the distance matrix for $H$ arbitrary and connected.
\end{theorem}

\begin{figure}[htb]
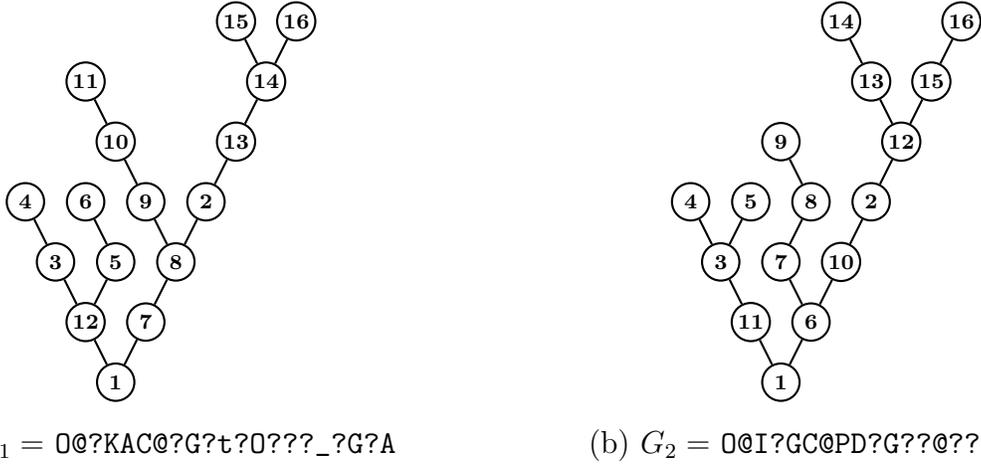
\centering
\begin{tabular}{c@{\hspace{1in}}c}
\PICCA & \PICCB \\[5pt]
(a) $G_1 =$ \verb'O@?KAC@?G?t?O???_?G?A'&
(b) $G_2 =$ \verb'O@I?GC@PD?G??@??_?_?@'
\end{tabular}
\caption{Graphs from McKay \cite{McKay}}
\label{fig:McKay}
\end{figure}

\begin{proof}
The graphs shown in Figure~\ref{fig:McKay} have the following similarity matrix for the distance matrices which satisfies the conditions of Theorem~\ref{thm:D}.
\[\scriptsize
(\,1\,)\oplus(\,1\,)\oplus\frac1{53}
\left(\begin{array}{rrrrrrrrrrrrrr}
20 & 7 & 7 & -15 & 16 & 1 & 3 & 2 & 15 & -2 & -6 & 0 & 19 & -14 \\
7 & 21 & 21 & 8 & -5 & 3 & 9 & 6 & -8 & -6 & 0 & 0 & -14 & 11 \\
20 & 7 & 7 & -15 & 16 & 1 & 3 & 2 & 15 & -2 & 19 & -14 & -6 & 0 \\
7 & 21 & 21 & 8 & -5 & 3 & 9 & 6 & -8 & -6 & -14 & 11 & 0 & 0 \\
-2 & -6 & -6 & 28 & 9 & -16 & 5 & 21 & 25 & -21 & 2 & 6 & 2 & 6 \\
15 & -8 & -8 & 2 & 12 & 14 & -11 & 28 & -2 & 25 & -15 & 8 & -15 & 8 \\
-16 & 5 & 5 & 12 & 19 & 31 & -13 & 9 & -12 & -9 & 16 & -5 & 16 & -5 \\
-1 & -3 & -3 & 14 & 31 & -8 & 29 & -16 & -14 & 16 & 1 & 3 & 1 & 3 \\
-3 & -9 & -9 & -11 & -13 & 29 & 34 & 5 & 11 & -5 & 3 & 9 & 3 & 9 \\
2 & 6 & 6 & 25 & -9 & 16 & -5 & -21 & 28 & 21 & -2 & -6 & -2 & -6 \\
-15 & 8 & 8 & -2 & -12 & -14 & 11 & 25 & 2 & 28 & 15 & -8 & 15 & -8 \\
33 & -7 & -7 & 15 & -16 & -1 & -3 & -2 & -15 & 2 & 20 & 7 & 20 & 7 \\
-7 & 0 & 11 & -8 & 5 & -3 & -9 & -6 & 8 & 6 & 7 & 21 & 7 & 21 \\
-7 & 11 & 0 & -8 & 5 & -3 & -9 & -6 & 8 & 6 & 7 & 21 & 7 & 21
\end{array}\right)
\]
So $G_1\coal_1H$ is cospectral with $G_2\coal_1H$ with respect to the distance matrix for $H$ arbitrary and connected.
\end{proof}

We note in passing that the graphs $G_1$ and $G_2$ in Figure~\ref{fig:McKay} are the same graph with different labelings.

Heysse \cite{Heysse} used coalescing to show that there exist pairs of cospectral graphs with arbitrarily many different numbers of edges. The key was finding a pair of cospectral graphs for $\mathcal{D}$ which had a vertex for which coalescing an arbitrary graph resulted in a cospectral graph. Heysse's proof is based on eigenvector arguments.

\begin{theorem}[{Heysse \cite{Heysse}}]
For the graphs shown in Figure~\ref{fig:Heysse}, $G_1\coal_1H$ is cospectral with $G_2\coal_1H$ with respect to the distance matrix  for $H$ arbitrary and connected.
\end{theorem}

\begin{figure}[htb]
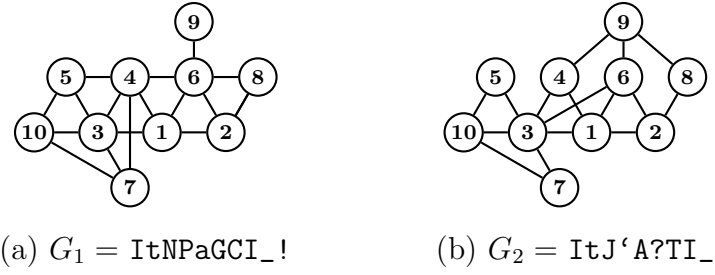
\centering
\begin{tabular}{c@{\hspace{0.75in}}c}
\PICDA & \PICDB \\[5pt]
(a) $G_1=$ \verb'ItNPaGCI_!'&
(b) $G_2=$ \verb'ItJ`A?TI_'
\end{tabular}
\caption{Graphs from Heysse \cite{Heysse}}
\label{fig:Heysse}
\end{figure}
\begin{proof}
The graphs shown in Figure~\ref{fig:Heysse} have the following similarity matrix for the distance matrices which satisfies the conditions of Theorem~\ref{thm:D}.
\[\small
(\,1\,)\oplus(\,1\,)\oplus\frac17
\left(\begin{array}{rrrrrrrr}
5 & 1 & 2 & 1 & 2 & 1 & -3 & -2 \\
-1 & 4 & 1 & 4 & 1 & -3 & 2 & -1 \\
2 & -1 & 1 & -1 & 2 & -1 & 3 & 2 \\
3 & 2 & -3 & 2 & -3 & 2 & 1 & 3 \\
2 & -1 & 2 & -1 & 1 & -1 & 3 & 2 \\
-1 & -3 & 1 & 4 & 1 & 4 & 2 & -1 \\
-1 & 4 & 1 & -3 & 1 & 4 & 2 & -1 \\
-2 & 1 & 2 & 1 & 2 & 1 & -3 & 5
\end{array}\right)
\]
So $G_1\coal_1H$ is cospectral with $G_2\coal_1H$ with respect to the distance matrix for $H$ arbitrary and connected.
\end{proof}

We note in passing that this proof shows that we could have also used vertex $2$ as the vertex to coalesce.

\subsubsection*{Small examples for  Theorem~\ref{thm:D} and Theorem~\ref{thm:genD}}

In Figure~\ref{fig:seven} are all eleven nonisomorphic pairs of graphs which are cospectral for $\mathcal{D}$ on seven vertices; there are no nonisomorphic graphs which are cospectral for $\mathcal{D}$ on six or fewer vertices. All pairs except (\verb'FqyWo', \verb'Ft@]o') are also cospectral for $\mathcal{D}^f$ for $f$ arbitrary.

\begin{figure}[htb]
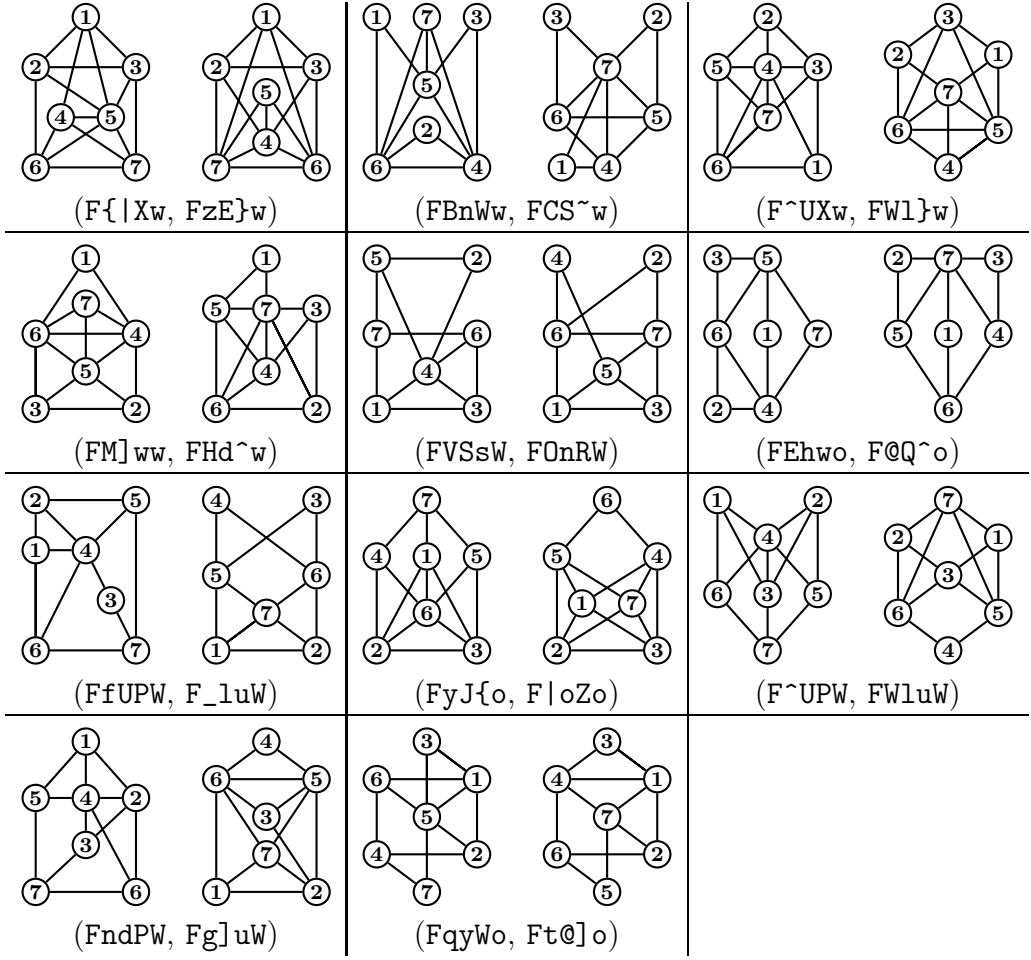

\centering
\begin{tabular}{c|c|c}
\PICSEVENA \\ 
(\verb'F{|Xw', \verb'FzE}w') & (\verb'FBnWw', \verb'FCS~w') & (\verb'F^UXw', \verb'FWl}w') \\ \hline  &&\\[-10pt]
\PICSEVENB \\
(\verb'FM]ww', \verb'FHd^w') & (\verb'FVSsW', \verb'FOnRW') & (\verb'FEhwo', \verb'F@Q^o') \\ \hline &&\\[-10pt]
\PICSEVENC \\ 
(\verb'FfUPW', \verb'F_luW') & (\verb'FyJ{o', \verb'F|oZo') & (\verb'F^UPW', \verb'FWluW') \\ \hline  &&\\[-10pt]
\PICSEVEND \\ 
(\verb'FndPW', \verb'Fg]uW') & (\verb'FqyWo', \verb'Ft@]o') & 
\end{tabular}
\cprotect\caption{All nonisomorphic graphs which are cospectral for $\mathcal{D}$ on seven vertices. All pairs \emph{except} (\verb'FqyWo', \verb'Ft@]o') are also cospectral for $\mathcal{D}^f$ for $f$ arbitrary.}
\label{fig:seven}
\end{figure}

For all of these graphs, the similarity matrix is
\[
S=
(\,1\,)\oplus(\,1\,)\oplus(\,1\,)\oplus\frac12
\left(\begin{array}{rrrrrrrr}
 -1 & 1 & 1 & 1 \\
 1 & -1 & 1 & 1 \\
 1 & 1 & -1 & 1 \\
 1 & 1 & 1 & -1
\end{array}\right),
\]
which means that for each pair we can glue an arbitrary $H_1$ onto vertex $1$, an arbitrary $H_2$ onto vertex $2$, an arbitrary $H_3$ onto vertex $3$, and an arbitrary $H_4$ onto the vertices $\{4,5,6,7\}$ in both graphs and the resulting pair of graphs will be cospectral for $\mathcal{D}$ and, except the final pair, $\mathcal{D}^f$ for $f$ arbitrary.

\subsubsection*{Assumption that $SJ=JS$ in Theorem~\ref{thm:D}}
From Theorem~\ref{thm:Lq} we can derive the following result for cospectral graphs with respect to $L^{(q)}$.

\begin{proposition}
If $G_1$ and $G_2$ are cospectral with respect to $L^{(q)}$, then $G_1\coal_VH$ is cospectral with $G_2\coal_VH$ for any $H$ with respect to $L^{(q)}$.
\end{proposition}

In other words, given two cospectral graphs if we coalesce the same graph onto each vertex the resulting graphs are cospectral.

\begin{proof}
Since they are cospectral there is a similarity matrix and we can treat the matrix as consisting of a single block which corresponds with all vertices. Now apply Theorem~\ref{thm:Lq}.
\end{proof}

One natural question is to ask if this can be extended to other matrices, and in particular the distance matrix. Small cases (e.g.\ up through eight) seem to indicate the answer is ``yes''. However, we can establish the following.

\begin{proposition}
If $G_1$ and $G_2$ are cospectral with respect to the distance matrix and there exists a similarity matrix $S$ for these distance matrices for which $SJ=JS$, then $G_1\coal_VH$ is cospectral with $G_2\coal_VH$ for any $H$ with respect to the distance matrix.
\end{proposition}
\begin{proof}
Since they are cospectral there is a similarity matrix and we can treat the matrix as consisting of a single block which corresponds with all vertices. Now apply Theorem~\ref{thm:D} combined with the assumption that $SJ=JS$.
\end{proof}

The assumption for $SJ=JS$ cannot be dropped as there do exist cospectral graphs with respect to the distance matrix with no such similarity matrices. For $n=9$ vertices there are $8$ such pairs out of the $14597$ pairs of cospectral graphs. These $8$ pairs are shown in Figure~\ref{fig:nine} and for these pairs there are \emph{no} subsets of vertices on which we can arbitrarily coalesce and maintain cospectrality for $\mathcal{D}$. For $n=10$ vertices there are $38$ such pairs out of the $875864$ pairs of cospectral graphs. This would seem to indicate that such pairs of cospectral graphs which do not have a similarity matrix that commutes with $J$ are relatively rare, which leads to the following questions.

\begin{question}
Do almost all pairs of cospectral graphs for the distance matrix have a similarity matrix $S$ for which $SJ=JS$?
\end{question}

\begin{question}
Are there constructions to form infinitely large families of pairs of cospectral graphs for the distance matrix that have no similarity matrix $S$ for which $SJ=JS$?
\end{question}

This last question might be particularly challenging as a common approach to forming large families of cospectral pairs is to start with small examples and then coalesce, but that is precisely what cannot happen for these graphs.

\begin{figure}[htb]
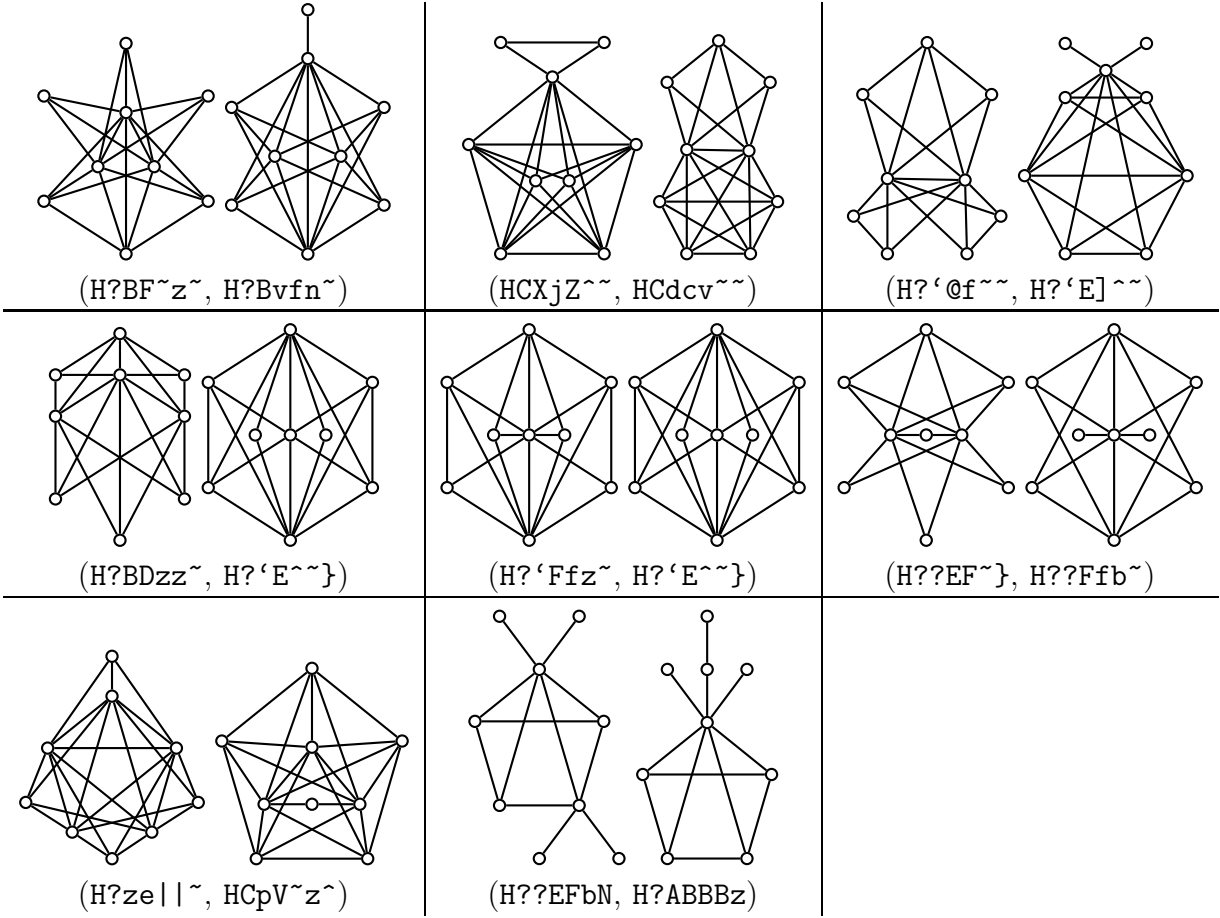

\centering
\begin{tabular}{c|c|c}
\PICNINEA \\ 
(\verb'H?BF~z~', \verb'H?Bvfn~') & (\verb'HCXjZ^~', \verb'HCdcv~~') & (\verb'H?`@f~~', \verb'H?`E]^~') \\ \hline & & \\[-10pt]
   \PICNINEB \\ 
(\verb'H?BDzz~', \verb'H?`E^~}') & (\verb'H?`Ffz~', \verb'H?`E^~}') & (\verb'H??EF~}', \verb'H??Ffb~') \\ \hline & & \\[-10pt]
\PICNINEC
\\ 
(\verb'H?ze||~', \verb'HCpV~z^') & (\verb'H??EFbN', \verb'H?ABBBz') & \\
\end{tabular}
\caption{All pairs of cospectral graphs on nine vertices with respect to the distance matrix which have no similarity matrix $S$ satisfying $SJ=JS$}
\label{fig:nine}
\end{figure}

\subsubsection*{Coalescing sets and unions for $\mathcal{D}$}
Sometimes no single block symmetric matrix captures all possible ways to coalesce graphs onto cospectral pairs. Among other things this shows that operations, such as unions, might not preserve the ability to coalesce. 

As an example, for the graphs shown in Figure~\ref{fig:unions} both of the following are similarity matrices for the distance matrices of the graphs.
\begin{align*}
S&\small=(\,1\,)\oplus(\,1\,)\oplus(\,1\,)\oplus(\,1\,)\oplus\frac12
\left(\begin{array}{rrrrrrrr}
 -1 & 1 & 1 & 1 \\
 1 & -1 & 1 & 1 \\
 1 & 1 & -1 & 1 \\
 1 & 1 & 1 & -1
\end{array}\right)
\\[5pt]
S&\small=
\frac12
\left(\begin{array}{rrrrrrrr}
-1 & 1 & 1 & 1 & 0 & 0 & 0 \\
1 & 1 & 0 & 0 & 0 & 1 & -1 \\
1 & 0 & 0 & 1 & 1 & -1 & 0 \\
1 & 0 & 1 & 0 & -1 & 0 & 1 \\
0 & 0 & -1 & 1 & 0 & 1 & 1 \\
0 & -1 & 1 & 0 & 1 & 1 & 0 \\
0 & 1 & 0 & -1 & 1 & 0 & 1
\end{array}\right)
\oplus(\,1\,)
\end{align*}
Between these two similarity matrices and Theorem~\ref{thm:D} this accounts for all ways to coalesce onto subsets of vertices and maintain cospectrality. In particular, we have that $G_1\coal_1H$ and $G_2\coal_1H$ are always cospectral for the distance matrix for arbitrary $H$; $G_1\coal_8H$ and $G_2\coal_8H$ are always cospectral for the distance matrix for arbitrary $H$; but $G_1\coal_{\{1,8\}}H$ and $G_2\coal_{\{1,8\}}H$ are not always cospectral for the distance matrix.

\begin{figure}[htb]
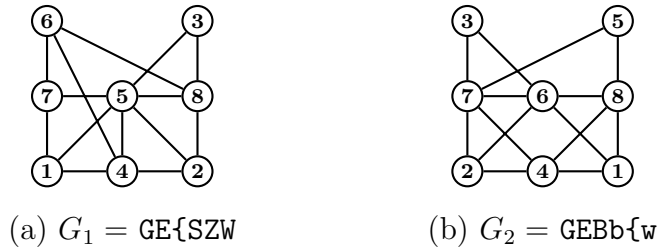
\centering
\begin{tabular}{c@{\hspace{1in}}c}
\PICE
\\[5pt]
(a) $G_1=$ \verb'GE{SZW'&(b) $G_2=$ \verb'GEBb{w'
\end{tabular}
\caption{Graphs which have two similarity matrices for the distance matrix with different block structures}
\label{fig:unions}
\end{figure}

\subsubsection*{Forming large cospectral graphs without cut vertices for the adjacency matrix}
Coalescing can be a tool to form large cospectral graphs, but the results always have cut vertices where the coalescing occurred. For the adjacency matrix we can get around this by using coalescing with the generalized distance matrix, $\mathcal{D}^f$. 

The idea is to use functions $f:0,1,\ldots \mapsto\{0,1\}$. In other words, we can take pairs of cospectral graphs for $\mathcal{D}^f$ and then add in edges for vertices which are given distances apart. The results will always be cospectral. As an example, in Figure~\ref{fig:genD} we have an example of two cospectral graphs and their decompositions (in the notation of Theorem~\ref{thm:genD}. The corresponding unions of any subset of these graphs are cospectral. The common similarity matrix for these graphs is
\[
S\small=(\,1\,)\oplus(\,1\,)\oplus(\,1\,)\oplus(\,1\,)\oplus(\,1\,)\oplus(\,1\,)\oplus(\,1\,)\oplus\frac12
\left(\begin{array}{rrrrrrrr}
 1 & 1 & -1 & 1 \\
 1 & -1 & 1 & 1 \\
 -1 & 1 & 1 & 1 \\
 1 & 1 & 1 & -1
\end{array}\right)
\]

\begin{figure}[htb]
\centering
\PICF
\cprotect\caption{Distance decompositions for $G_1=$ \verb'JCO_?c]@_S?' and $G_2=$ \verb'JCO_?sAB_k?'}
\label{fig:genD}
\end{figure}

In general, start with a pair of graphs that are cospectral for $\mathcal{D}^f$, then build large graphs which are cospectral for $\mathcal{D}^f$ by coalescing. Finally, take unions of various distances to eliminate cut vertices. We also note that for any such resulting graphs that the complements will also be cospectral (by taking the union of the complements of the distances).

\subsection*{Acknowledgment} This research was conducted primarily at the 2023 Iowa State University Math REU which was supported through NSF Grant DMS-1950583.

\bibliographystyle{plain}
\bibliography{bibliography}
\end{document}